\documentclass{amsart}
\usepackage{amssymb}
\usepackage[mathcal]{euscript}
\usepackage[cmtip,all]{xy}

\newcommand{\bydef}{:=}

\newcommand{\id}{\mathrm{id}}

\newcommand{\trace}{\mathrm{trace}}

\newcommand{\cC}{\mathcal{C}}

\newcommand{\cL}{\mathcal{L}}

\newcommand{\cU}{\mathcal{U}}
\newcommand{\cV}{\mathcal{V}}

\newcommand{\cX}{\mathfrak{X}}

\newcommand{\frg}{{\mathfrak g}}
\newcommand{\frh}{{\mathfrak h}}

\newcommand{\frmm}{{\mathfrak m}} 

\newcommand{\ZZ}{\mathbb{Z}}

\newcommand{\RR}{\mathbb{R}}
\newcommand{\CC}{\mathbb{C}}

\newcommand{\bS}{\mathbb{S}}

\DeclareMathOperator{\Hom}{\mathrm{Hom}}
\DeclareMathOperator{\End}{\mathrm{End}}

\DeclareMathOperator{\Aut}{\mathrm{Aut}}
\DeclareMathOperator{\Der}{\mathrm{Der}}

\DeclareMathOperator{\Mat}{\mathrm{Mat}}

\newcommand{\ad}{\mathrm{ad}}
\newcommand{\Ad}{\mathrm{Ad}}

\newcommand{\frgl}{{\mathfrak{gl}}}

\newcommand{\frsu}{{\mathfrak{su}}}

\newcommand{\GL}{\mathrm{GL}}

\newcommand{\subo}{_{\bar 0}}
\newcommand{\subuno}{_{\bar 1}}

\newcommand{\bL}{\boldsymbol{L}}

\newtheorem{theorem}{Theorem}
\newtheorem{proposition}[theorem]{Proposition}
\newtheorem{lemma}[theorem]{Lemma}
\newtheorem{corollary}[theorem]{Corollary}

\theoremstyle{definition}
\newtheorem{definition}[theorem]{Definition}
\newtheorem{example}[theorem]{Example}
\newtheorem{examples}[theorem]{Examples}

\theoremstyle{remark}
\newtheorem{remark}[theorem]{Remark}
\newtheorem{exercise}{Exercise}

\newenvironment{romanenumerate}
 {\begin{enumerate}
 
 }{\end{enumerate}}

\begin{document}

\title[Homogeneous spaces and nonassociative algebras]{Reductive homogeneous spaces\\ 
and nonassociative algebras}

\author[Alberto Elduque]{Alberto Elduque${}^\star$}
\address{Departamento de Matem\'{a}ticas
 e Instituto Universitario de Matem\'aticas y Aplicaciones,
 Universidad de Zaragoza, 50009 Zaragoza, Spain}
\email{elduque@unizar.es}
\thanks{${}^\star$Supported by grants MTM2017-83506-C2-1-P (AEI/FEDER, UE) and E22\_17R 
(Gobierno de Arag\'on, Grupo de referencia ``\'Algebra y
Geometr{\'\i}a'', cofunded by Feder 2014-2020 ``Construyendo Europa desde Arag\'on'')}



\date{}

\begin{abstract}
The purpose of these survey notes is to give a presentation of a classical theorem of Nomizu \cite{Nom54} that relates the invariant affine connections on reductive homogeneous spaces and nonassociative algebras.
\end{abstract}

\maketitle


These notes have been around for about 20 years. They were polished on the occasion of a CIMPA research school, held in Marrakech (April 13-24, 2015), where an initial version became the notes of a course with the same title in this school. The title of the school was `G\'eom\'etrie diff\'erentielle et alg\`ebres non associatives'. The author is indebted to the participants in the course for their questions and comments, that improved the initial exposition.

The purpose of the notes is to present a classical result by Nomizu \cite{Nom54} that relates the invariant affine connections on reductive homogeneous spaces to nonassociative algebras defined on the tangent space of a point. This allows us to give precise algebraic descriptions of these connections, of their torsion and curvature tensors, ... 

Nomizu's result constitutes a very nice bridge between these two areas mentioned in the title of the school.

Modulo basic results, the presentation is self contained. The first three sections review the results on smooth manifolds, on affine connections and on Lie groups and Lie algebras that will be needed for our purpose. The fourth section is devoted to study the invariant affine connections on homogeneous spaces, following the presentation in 
\cite[Chapter IV]{GeometryI}. 

Nomizu's Theorem is presented and proved in Section \ref{se:Nomizu}, based on the previous work. The proof given by Nomizu in his seminal paper is quite different. Also, Lie-Yamaguti algebras will be introduced in Section \ref{se:Nomizu}, Exercise \ref{ex:LY}. Finally, in Section \ref{se:examples} left and bi-invariant affine connections in Lie groups will be studied. Several classes of algebras, like 
Lie-admissible, flexible, associative or left-symmetric, appear naturally.

Several exercises, complementing the theory, are scattered through the text. 

The material presented is well known (see \cite{GeometryI,KoNo}). It is expected that this presentation will be useful to the readers, and will whet their appetite to learn more about the subject.

\section{Manifolds}\label{se:manifolds}

\subsection{Manifolds} In these notes an $n$-dimensional \emph{manifold} will indicate a second countable Hausdorff topological space $M$ covered by a family of open sets $\{\cU_\lambda:\lambda\in\Lambda\}$, together with coordinate maps
\[
\begin{split}
\phi_\lambda:\cU_\lambda&\longrightarrow \RR^n\\
p&\mapsto \bigl(x_1(p),x_2(p),\ldots,x_n(p)\bigr),
\end{split}
\]
such that $\phi_\lambda$ is a homeomorphism of $\cU_\lambda$ into its range, an open set in $\RR^n$, satisfying that
\[
\phi_\lambda\circ\phi_\mu^{-1}: \phi_\mu(\cU_\lambda\cap\cU_\mu)\rightarrow
\phi_\lambda(\cU_\lambda\cap\cU_\mu)
\]
is a $C^\infty$-map (i.e., all the partial derivatives of any order exist and are continuous). The family of ``charts'' $\{(\cU_\lambda,\phi_\lambda):\lambda\in\Lambda\}$ is called an \emph{atlas} of the manifold.

The reader may consult \cite{Conlon}, \cite{Warner} or \cite{GeometryI} for the basic facts in this section.

\begin{examples}
\begin{itemize}
\item The euclidean space $\RR^n$ with the trivial atlas $\{(\RR^n,\id)\}$. (Or any open subset of $\RR^n$.)
\item The unit circle $\bS^1=\{(x,y)\in\RR^2:x^2+y^2=1\}$ with the atlas consisting of the following four charts:
\begin{align*}
\cU_{\textrm{top}}&=\{(x,y)\in\bS^1: y>0\}&\quad\phi_{\textrm{top}}:\cU_{\textrm{top}}&\rightarrow \RR,\ (x,y)\mapsto\arccos x,\\
\cU_{\textrm{bottom}}&=\{(x,y)\in\bS^1: y<0\}&\quad\phi_{\textrm{bottom}}:\cU_{\textrm{bottom}}&\rightarrow \RR,\ (x,y)\mapsto\arccos x,\\
\cU_{\textrm{right}}&=\{(x,y)\in\bS^1: x>0\}&\quad\phi_{\textrm{right}}:\cU_{\textrm{right}}&\rightarrow \RR,\ (x,y)\mapsto\arcsin y,\\
\cU_{\textrm{left}}&=\{(x,y)\in\bS^1: x<0\}&\quad\phi_{\textrm{left}}:\cU_{\textrm{left}}&\rightarrow \RR,\ (x,y)\mapsto\arcsin y.
\end{align*}
Here we assume $\arccos: (-1,1)\rightarrow (0,\pi)$ and $\arcsin:(-1,1)\rightarrow \left(-\frac{\pi}{2},\frac{\pi}{2}\right)$.

Note that, for example, $\cU_{\textrm{top}}\cap\cU_{\textrm{left}}=\{(x,y)\in\bS^1: y>0,x<0\}$ and the change of coordinates is the map:
\[
\begin{split}
\phi_{\textrm{top}}\circ\phi_{\textrm{left}}^{-1}:\left(0,\frac{\pi}{2}\right)&\rightarrow
\left(\frac{\pi}{2},\pi\right)\\
t\quad&\mapsto\ \pi-t,
\end{split}
\]
which is clearly $C^\infty$.

\item The cartesian product $M\times N$ of an $m$-dimensional manifold $M$ and an $n$-dimensional manifold $N$ is naturally an $(m+n)$-dimensional manifold.
\end{itemize}
\end{examples}

A continuous map $\Psi:M\rightarrow N$ between two manifolds $M$ and $N$ of respective dimensions $m$ and $n$ is \emph{smooth} if for any charts $(\cU_\lambda,\phi_\lambda)$ and $(\cV_\mu,\varphi_\mu)$ of $M$ and $N$ respectively, the map (between open sets in $\RR^m$ and $\RR^n$):
\[
\varphi_\mu\circ\Psi\circ\phi_\lambda^{-1}:\phi_\lambda\left(\cU_\lambda\cap\Psi^{-1}(\cV_\mu)\right)\rightarrow \varphi_\mu(\cV_\mu)
\]
is of class $C^\infty$.

In particular, we will consider the set $\cC^\infty(M)$ of smooth maps from the manifold $M$ into $\RR$. This is a unital, commutative, associative ring with the pointwise addition and multiplication of maps. Actually, it is an algebra over $\RR$, as we can multiply maps by constants.

\bigskip

\subsection{Vector fields} A \emph{nonassociative real algebra} is a real vector space $A$ endowed with a bilinear map (multiplication) $m:A\times A\rightarrow A$. We will usually (but not always!) write $xy$ for $m(x,y)$. The notions of homomorphism, subalgebra, ideal, ..., are the natural ones. 
Note that the word \emph{nonassociative} simply means `not necessarily associative'.

For example, a \emph{Lie algebra} is a nonassociative algebra $L$ with multiplication $(x,y)\mapsto [x,y]$ such that
\begin{itemize}
\item $[x,x]=0$ for any  $x\in L$ (the product is \emph{anticommutative}),
\item $[[x,y],z]+[[y,z],x]+[[z,x],y]=0$ for any $x,y,z\in L$ (\emph{Jacobi identity}).
\end{itemize}

Given a nonassociative algebra $A$ and a homomorphism $\varphi:A\rightarrow \RR$, a \emph{$\varphi$-derivation} is a linear map $d:A\rightarrow \RR$ such that $d(xy)=d(x)\varphi(y)+\varphi(x)d(y)$ for any $x,y\in A$. The set of $\varphi$-derivations forms a vector space. 

On the other hand, a \emph{derivation} of $A$ is a linear map $D:A\rightarrow A$ such that $D(xy)=D(x)y+xD(y)$ for any $x,y\in A$. 
The set of derivations $\Der(A)$ is a Lie algebra with the usual `bracket': $[D_1,D_2]=D_1\circ D_2-D_2\circ D_1$.

\begin{exercise} Check that indeed, if $D_1$ and $D_2$ are derivations of $A$, so is $[D_1,D_2]$.
\end{exercise}

\medskip

Given a manifold $M$ and a point $p\in M$, the map $\varphi_p:\cC^\infty(M)\rightarrow \RR$ given by evaluation at $p$: $\varphi_p(f)\bydef f(p)$, is a homomorphism. The \emph{tangent space} $T_pM$ at $p$ is the vector space of $\varphi_p$-derivations of $\cC^\infty(M)$\footnote{There are other possible definitions, but this is suitable for our purposes.}. That is, the elements of $T_pM$ (called \emph{tangent vectors} at $p$) are linear maps $v_p:\cC^\infty(M)\rightarrow \RR$, such that 
\[
v_p(fg)=v_p(f)g(p)+f(p)v_p(g)
\]
for any $f,g\in\cC^\infty(M)$.

Given a chart $(\cU,\phi)$ of $M$ with $p\in \cU$ and $\phi(q)=\bigl(x_1(q),\ldots,x_n(q)\bigr)$ for any $q\in\cU$,  we have the natural tangent vectors $\left.\frac{\partial\ }{\partial x_i}\right|_p$ given by
\[
f\mapsto \frac{\partial(f\circ\phi^{-1})}{\partial x_i}\bigl(\phi(p)\bigr).
\]
Actually, the elements $\left.\frac{\partial\ }{\partial x_i}\right|_p$, $i=1,\ldots,n$, form a basis of $T_pM$.

\medskip

The disjoint union 
\[
TM\bydef \bigcup_{p\in M}T_pM
\]
is called the \emph{tangent bundle} of $M$. It has a natural structure of manifold with atlas $\{(\tilde\cU_\lambda,\tilde\phi_\lambda):\lambda\in \Lambda\}$, where $\tilde\cU_\lambda=\bigcup_{p\in \cU_\lambda}T_pM$ and
\[
\begin{split}
\tilde\phi_\lambda:\tilde\cU_\lambda&\longrightarrow \RR^{2n}=\RR^n\times\RR^n\\
\sum_{i=1}^n\alpha_i\left.\frac{\partial\ }{\partial x_i}\right|_p\in T_pM
&\mapsto \bigl(\phi_\lambda(p),(\alpha_1,\ldots,\alpha_n)\bigr).
\end{split}
\]

A \emph{vector field} of a manifold $M$ is a smooth section of the natural projection $\pi:TM\rightarrow M$, $v_p\in T_pM\mapsto p$. That is, a vector field is a smooth map
\[
\begin{split}
X:M&\longrightarrow TM\\
p&\mapsto \ X_p,
\end{split}
\]
such that $X_p\in T_pM$ for any $p\in M$. Any vector field $X$ induces a derivation of the algebra $\cC^\infty(M)$ as follows: for any $f\in\cC^\infty(M)$, the image of $f$ under this derivation, also denoted by $X$, is the smooth map 
\[
Xf: p\mapsto X_p(f).
\] 
Any derivation of $\cC^\infty(M)$ is obtained in this way from a vector field. Therefore, we will identify the set of vector fields, denoted by $\cX(M)$, with the Lie algebra $\Der\bigl(\cC^\infty(M)\bigr)$.

\begin{exercise}\label{exe:RDer}
Let $R$ be a commutative associative algebra and let $\cL=\Der(R)$ be its Lie algebra of derivations (keep in mind the case $R=\cC^\infty(M)$!).
\begin{itemize}
\item Prove that $\cL$ is a module over $R$ by means of $(aD)(b)=aD(b)$ for any $a,b\in R$.
\item Prove that $[aD,bE]=ab[D,E]+\bigl(aD(b)\bigr)E-\bigl(bE(a)\bigr)D$ for any $a,b\in R$ and $D,E\in \cL$.
\end{itemize}
\end{exercise}

\bigskip

Let $\Phi:M\rightarrow N$ be a smooth map between two manifolds, and let $v_p\in T_pM$ be a tangent vector at a point $p\in M$. Then we can push forward $v_p$ to a tangent vector $\Phi_*(v_p)$ defined by $\Phi_*(v_p)(f)=v_p(f\circ\Phi)$ for any $f\in\cC^\infty(N)$. The `tangent map' $\Phi_*:TM\rightarrow TN$, $v_p\in T_pM\mapsto \Phi_*(v_p)\in T_{\Phi(p)}N$ is smooth.

In particular, if $\gamma:(-\epsilon,\epsilon)\rightarrow M$ is a smooth map with $\gamma(0)=p$ (a curve centered at $p$), then 
\[
\dot\gamma(0)\bydef \gamma_*\left(\left.\frac{d\ }{dt}\right|_{t=0}\right): f\mapsto 
\frac{d(f\circ\gamma)}{dt}(0)
\]
is a tangent vector at $p\in M$. Any tangent vector is obtained in this way.

\smallskip

Two vector fields $X\in\cX(M)$ and $Y\in\cX(N)$ are said to be \emph{$\Phi$-related} if $\Phi_*(X_p)=Y_{\Phi(p)}$ for any $p\in M$. In other words, $X$ and $Y$ are $\Phi$-related if
\[
(Yf)\circ\Phi=X(f\circ\Phi)
\]
for any $f\in\cC^\infty(N)$.

\begin{exercise}\label{exe:Phi_related}
Consider vector fields $X_1,X_2\in\cX(M)$ and $Y_1,Y_2\in\cX(N)$ such that $X_i$ and $Y_i$ are $\Phi$-related, for $i=1,2$. Prove that $[X_1,X_2]$ and $[Y_1,Y_2]$ are also $\Phi$-related.
\end{exercise}

If the smooth map $\Phi$ above is a diffeomorphism (i.e.; it is smooth, bijective, and the inverse is also smooth), then $\Phi$ induces an isomorphism of Lie algebras (denoted by $\Phi_*$ too):
\[
\begin{split}
\Phi_*:\cX(M)&\longrightarrow \cX(N)\\
X\ &\mapsto \Phi_*(X): f\in\cC^\infty(N)\mapsto X(f\circ\Phi)\circ\Phi^{-1}\in\cC^\infty(N).
\end{split}
\]
In this case, the vector fields $X\in\cX(M)$ and $Y\in\cX(M)$ are $\Phi$-related if and only if $Y=\Phi_*(X)$.

\bigskip

\subsection{Flows} 

Let $X\in\cX(M)$ be a vector field on a manifold $M$. For any $p\in M$ there exists a positive real number $\epsilon >0$, a neighborhood $U$ of $p$ in $M$, and a smooth map
\[
\begin{split}
\Phi:(-\epsilon,\epsilon)\times U&\longrightarrow M\\
(t,m)\ &\mapsto \Phi(t,m),
\end{split}
\]
such that for any $f\in\cC^\infty(M)$ and any $m\in U$ we have
$\Phi(0,m)=m$, and
\begin{equation}\label{eq:XmfPhi}
X_{\Phi(t,m)}f=\frac{d\ }{dt}f\bigl(\Phi(t,m)\bigr)\,\left(=\lim_{s\rightarrow 0}
\frac{f\bigl(\Phi(t+s,m)\bigr)-f\bigl(\Phi(t,m)\bigr)}{s}\right).
\end{equation}
The map $\Phi$ is a \emph{local flow} of $X$ at $p$. 

Local flows always exist by the fundamental existence and uniqueness theorem for first-order differential equations, which also implies that 
\begin{equation}\label{eq:Phi_t1t2}
\Phi(t_1+t_2,m)=\Phi(t_1,\Phi(t_2,m))
\end{equation}
when this makes sense.

The vector field $X$ is said to be \emph{complete} if there exists a \emph{global flow}
\[
\Phi:\RR\times M\rightarrow M,
\]
so that \eqref{eq:XmfPhi} holds for any $m\in M$ and $t\in\RR$. (Note that, conversely, any smooth map $\Phi:(-\epsilon,\epsilon)\times M\rightarrow M$ satisfying \eqref{eq:Phi_t1t2} determines a vector field by equation \eqref{eq:XmfPhi}.)

\smallskip

If $X$ is a complete vector field with global flow $\Phi$, for any $t\in \RR$ the map $\Phi_t:M\rightarrow M$ given by $\Phi_t(m)=\Phi(t,m)$ is a diffeomorphism and, for any vector field $Y\in\cX(M)$, the Lie bracket $[X,Y]$ can be computed as follows (see \cite[Theorem 4.3.2]{Conlon}):
\begin{equation}\label{eq:XYPhi}
[X,Y]=\left.\frac{d\ }{dt}\right|_{t=0}(\Phi_{-t})_*(Y)=
\lim_{t\rightarrow 0}\frac{(\Phi_{-t})_*(Y)-Y}{t}.
\end{equation}
Thus, the value at a point $p$: $[X,Y]_p$ is given by
\[
[X,Y]_p=\lim_{t\rightarrow 0}\frac{(\Phi_{-t})_*(Y_{\Phi_t(p)})-Y_p}{t}
\]
(a limit in the vector space $T_pM$). Actually, for $f\in\cC^\infty(M)$, consider the smooth map 
\[
\begin{split}
H:\RR\times\RR&\longrightarrow \RR\\
(t,s)\,&\mapsto H(t,s)=Y_{\Phi_t(p)}(f\circ\Phi_s).
\end{split}
\]
Then,
\[
\begin{split}
\frac{\partial H}{\partial t}(0,0)&=\left.\frac{d\ }{dt}\right|_{t=0}Y_{\Phi_t(p)}(f)=
   \left.\frac{d\ }{dt}\right|_{t=0}(Yf)(\Phi_t(p))=X_p(Yf),\\[6pt]
\frac{\partial H}{\partial s}(0,0)&=\left.\frac{d\ }{ds}\right|_{s=0}Y_{p}(f\circ\Phi_s)
   =Y_p\left(\left.\frac{d\ }{ds}\right|_{s=0}(f\circ\Phi_s)\right)\\
  &\qquad\qquad\textrm{(by the equality of mixed partial derivatives)}\\
  &=Y_p(Xf).
\end{split}
\]
Hence we obtain
\[
\left.\frac{d\ }{dt}\right|_{t=0}Y_{\Phi_t(p)}(f\circ\Phi_{-t})=
\left.\frac{d\ }{dt}\right|_{t=0}H(t,-t)=X_p(Yf)-Y_p(Xf)=[X,Y]_p(f),
\]
as indicated in \eqref{eq:XYPhi}.

Two vector fields $X,Y\in\cX(M)$ commute (i.e.; $[X,Y]=0$) if and only if the corresponding flows  commute: for any $m\in M$ there is a $\delta_m>0$ such that $\Phi_t\circ \Psi_s(m)=\Psi_s\circ\Phi_t(m)$  for $-\delta_m<t,s<\delta_m$ \cite[Theorem 2.8.20]{Conlon}

\bigskip

\section{Affine connections}\label{se:affine_connections}

Affine connections constitute the extension to ``non flat'' manifolds of the idea of directional derivative of vector fields. It provides the idea of ``parallel transport''.

\begin{definition} An \emph{affine connection} on a manifold $M$ is an $\RR$-bilinear map
\[
\begin{split}
\nabla:\cX(M)\times\cX(M)&\longrightarrow \cX(M)\\
(X,Y)\quad &\mapsto\ \nabla_XY,
\end{split}
\]
with the following properties:
\begin{enumerate}
\item $\nabla$ is $\cC^\infty(M)$-linear in the first component: 
\[
\nabla_{fX}Y=f\nabla_XY
\] 
for any $X,Y\in\cX(M)$ and any $f\in\cC^\infty(M)$. (Recall from Exercise \ref{exe:RDer} that $\cX(M)$ is a module for $\cC^\infty(M)$.)

\item  For any $X,Y\in\cX(M)$ and any $f\in\cC^\infty(M)$, the equation
\[
\nabla_X(fY)=(Xf)Y+f\nabla_XY
\]
holds.
\end{enumerate}
\end{definition}

The first property implies that for any $X,Y\in\cX(M)$ and any point $p\in M$, $\bigl(\nabla_XY\bigr)_p$ depends only on $X_p$ and $Y$, so we may define $\nabla_vY\in T_pM$ for any  $v\in T_pM$ and $Y\in\cX(M)$, which we should think as the derivative of the vector field $Y$ in the direction $v$ at the point $p$.

Also, the second property shows that $\bigl(\nabla_XY\bigr)_p$ depends on the values of $Y$ in a neighborhood of $p$. Thus affine connections can be restricted to open subsets of $M$, in particular to domains of charts. If $(U,\phi)$ is a chart of $M$ with $\phi(q)=\bigl(x_1(q),\ldots,x_n(q)\bigr)$, we obtain the $\cC^\infty(U)$-basis of the vector fields in $\cX(U)$: $\left\{\frac{\partial\ }{\partial x_1},\cdots,\frac{\partial\ }{\partial x_n}\right\}$. By properties (1) and (2), the restriction of the affine connection $\nabla$ on $M$ to $U$ is determined by the values
\[
\nabla_{\frac{\partial\ }{\partial x_i}}\frac{\partial\ }{\partial x_j}
=\sum_{k=1}^n\Gamma_{ij}^k\frac{\partial\ }{\partial x_k},
\]
for some maps $\Gamma_{ij}^k\in\cC^\infty(U)$. These maps $\Gamma_{ij}^k$ are called the \emph{Christoffel symbols} of $\nabla$ in the chart $(U,\phi)$.

Then, for $X,Y\in\cX(M)$, their restrictions to $U$ are of the form $X\vert_U=\sum_{i=1}^nf_i\frac{\partial\ }{\partial x_i}$ and $Y\vert_U=\sum_{i=1}^ng_i\frac{\partial\ }{\partial x_i}$ ($f_i,g_i\in\cC^\infty(U)$ for $i=1,\ldots,n$), and the vector field $\bigl(\nabla_XY\bigr)\vert_U$ is given by:
\begin{equation}\label{eq:nablaXY_U}
\bigl(\nabla_XY\bigr)\vert_U=\sum_{k=1}^n\left(\sum_{i=1}^nf_i\frac{\partial g_k}{\partial x_i} +\sum_{i,j=1}^nf_ig_j\Gamma_{ij}^k\right)\frac{\partial\ }{\partial x_k}.
\end{equation}

\begin{definition} Given an affine connection $\nabla$ on a manifold $M$, for any vector field $X\in\cX(M)$ the \emph{Nomizu operator}
\[
\bL_X:\cX(M)\rightarrow \cX(M)
\]
is the $\cC^\infty(M)$-linear form such that
\[
\bL_XY=\nabla_XY-[X,Y]
\]
for any $Y\in\cX(M)$.
\end{definition}

\begin{exercise}
Check that $\bL_X$ is indeed $\cC^\infty(M)$-linear.

More generally, let $\phi$ be a unital commutative and associative ring, $R$ a commutative associative $\phi$-algebra and $D=\Der(R)$ its Lie algebra of derivations (this is a module over $R$ by Exercise \ref{exe:RDer}). Define an \emph{affine connection} $\nabla$ on $(R,D)$ as a $\phi$-bilinear map
\[
\begin{split}
\nabla: D\times D&\longrightarrow D\\
(\delta_1,\delta_2)&\mapsto \nabla_{\delta_1}\delta_2,
\end{split}
\]
subject to:
\begin{itemize}
\item $\nabla_{r\delta_1}\delta_2=r\nabla_{\delta_1}\delta_2$,
\item $\nabla_{\delta_1}(r\delta_2)=\delta_1(r)\delta_2+r\nabla_{\delta_1}\delta_2$,
\end{itemize}
for any $r\in R$ and $\delta_1,\delta_2\in D$. Prove that for $\delta\in D$, the \emph{Nomizu operator} $\bL_{\delta}:D\rightarrow D$, $\delta'\mapsto \nabla_\delta\delta'-[\delta,\delta']$, is $R$-linear.
\end{exercise}

\smallskip

The $\cC^\infty(M)$-linearity of $\bL_X$ implies that for any $q\in M$, the tangent vector $\bigl(\bL_XY\bigr)_q$ depends only on $Y_q$, so we have a well defined endomorphism $(\bL_X)\vert_q:T_qM\rightarrow T_qM$. Also, $(\bL_X)\vert_q$ depends only on the values of $X$ in a neighborhood of $q$.

\medskip

In order to define what we understand by ``parallel transport'' we need some preliminaries.

\begin{definition} Let $a,b\in\RR$, $a<b$, and let $\gamma:[a,b]\rightarrow M$ be a smooth map into a manifold $M$ (this means that there is $\epsilon>0$ such that $\gamma$ extends to a smooth map $(a-\epsilon,b+\epsilon)\rightarrow M$). A \emph{vector field} along the curve $\gamma$ is a smooth map $\nu:[a,b]\rightarrow TM$ such that the diagram
\[
\begin{gathered}
\xymatrix{
&TM \ar[d]^\pi \\
[a,b] \ar[r]^\gamma \ar[ru]^\nu &M}
\end{gathered}
\]
commutes.
\end{definition}

We will denote by $\cX(\gamma)$ the vector space of vector fields along $\gamma$. 

For example, the derivative $\dot\gamma$ given by:
\[
\dot\gamma(t)\bydef \gamma_*\left(\frac{d\ }{dt}\right)\,\in T_{\gamma(t)}M
\]
is a vector field along $\gamma$. If $(U,\phi)$ is a chart such that the image of $\gamma$ is contained in $U$ and $\phi(\gamma(t))=(\gamma_1(t),\ldots,\gamma_n(t))$, then 
\[
\dot\gamma(t)=\sum_{i=1}^n\gamma_i'(t)\left(\frac{\partial\ }{\partial x_i}\right)_{\gamma(t)}.
\]

Given an affine connection $\nabla$ on the manifold $M$, a curve $\gamma:[a,b]\rightarrow M$ and a vector field $Y\in\cX(M)$, in local coordinates we can express $Y$ as
\[
Y=\sum_{i=1}^ng_i\frac{\partial\ }{\partial x_i},
\]
for smooth maps $g_i$. Then we can compute
\[
\begin{split}
\nabla_{\dot\gamma(t)}Y
 &=\sum_{i=1}^n\nabla_{\dot\gamma(t)}\left(g_i\frac{\partial\ }{\partial x_i}\right)\\
 &=\sum_{i=1}^n\left(\dot\gamma(t)(g_i)\left.\frac{\partial\ }{\partial x_i}\right|_{\gamma(t)} +
   g_i(\gamma(t))\sum_{k=1}^n\gamma_j'(t)\Gamma_{ji}^k(\gamma(t))\left.\frac{\partial\ }{\partial x_k}\right|_{\gamma(t)}\right)\\
  &=\sum_{k=1}^n\left(\frac{d g_k(\gamma(t))}{dt}
     +\sum_{i,j=1}^n\gamma_j'(t)g_i(\gamma(t))\Gamma_{ji}^k(\gamma(t))\right)
     \left.\frac{\partial\ }{\partial x_k}\right|_{\gamma(t)},
\end{split}
\]
and this depends only on the values of $Y$ in the points $\gamma(t)$. This allows us, given a vector field along $\gamma$: $\nu\in\cX(\gamma)$, which in local coordinates appears as $\nu(t)=\sum_{i=1}^n\nu_i(t)\left.\frac{\partial\ }{\partial x_i}\right|_{\gamma(t)}$, to define 
\[
\nabla_{\dot\gamma(t)}\nu\in T_{\gamma(t)}M,
\]
by means of
\[
\nabla_{\dot\gamma(t)}\nu=
\sum_{k=1}^n\left(\frac{d\nu_k(t)}{dt}
  +\sum_{i,j=1}^n\gamma_j'(t)\nu_i(t)\Gamma_{ji}^k(\gamma(t))\right)
  \left.\frac{\partial\ }{\partial x_k}\right|_{\gamma(t)} \,\in T_{\gamma(t)}M.
\]
In this way we obtain an operator
\[
\begin{split}
\cX(\gamma)&\longrightarrow \cX(\gamma)\\
\nu&\mapsto \Bigl(t\mapsto \nabla_{\dot\gamma(t)}\nu\Bigr).
\end{split}
\]

\begin{definition}
With $M$ and $\gamma$ as above, a vector field along $\gamma$: $\nu\in\cX(\gamma)$, is said to be \emph{parallel} if $\nabla_{\dot\gamma(t)}\nu=0$ for any $a\leq t\leq b$.
\end{definition}

The existence and uniqueness of solutions for ordinary differential equations prove that for any $a\leq c\leq b$, if $v_0\in T_{\gamma(c)}M$, then there is a unique parallel vector field $\nu$ along $\gamma$ such that $\nu_{\gamma(c)}=v_0$. This field $\nu$ is called the \emph{parallel transport of $v_0$ along $\gamma$}.

\begin{definition}
The curve $\gamma$ is said to be a \emph{geodesic} if $\dot\gamma$ is parallel along $\gamma$.
\end{definition}

\medskip

There are two important tensors attached to any affine connection on a manifold.

\begin{definition}
Let $\nabla$ be an affine connection on a manifold $M$.
\begin{itemize}
\item The \emph{torsion tensor} of $\nabla$ is the $\cC^\infty(M)$-bilinear map:
\[
\begin{split}
T:\cX(M)\times\cX(M)&\longrightarrow \cX(M)\\
 (X,Y)\quad &\mapsto\ T(X,Y)\bydef \nabla_XY-\nabla_YX-[X,Y].
\end{split}
\]
If the torsion tensor is identically $0$, then $\nabla$ is said to be \emph{symmetric} (or \emph{torsion-free}).

\smallskip

\item The \emph{curvature tensor} of $\nabla$ is the $\cC^\infty(M)$-trilinear map:
\[
\begin{split}
R:\cX(M)\times\cX(M)\times\cX(M)&\longrightarrow \cX(M)\\
 (X,Y,Z)\quad &\mapsto\  R(X,Y)Z\bydef \nabla_X\nabla_YZ-\nabla_Y\nabla_XZ-\nabla_{[X,Y]}Z.
\end{split}
\]
If the curvature tensor is identically $0$, then $\nabla$ is said to be \emph{flat}.
\end{itemize}
\end{definition}

\begin{exercise} Check that the torsion and curvature tensors are indeed $\cC^\infty(M)$-linear in each component. Define these concepts in the general setting of Exercise \ref{exe:RDer}.
\end{exercise}

The $\cC^\infty(M)$-linearity implies that the values of $T(X,Y)$ or $R(X,Y)Z$ at a given point $p\in M$ depend only on the values of the vector fields $X$, $Y$ and $Z$ at this point. Hence it makes sense to consider the torsion $T(u,v)$ or curvature $R(u,v)w$ for $u,v,w\in T_pM$.

\begin{exercise}
Prove that an affine connection $\nabla$ is symmetric if its Christoffel symbols in any chart satisfy $\Gamma_{ij}^k=\Gamma_{ji}^k$ for any $1\leq i,j,k\leq n$.
\end{exercise}

\begin{exercise} A \emph{pseudo-Riemannian} manifold is a manifold $M$ endowed with a nondegenerate symmetric $\cC^\infty(M)$-bilinear form (the \emph{pseudo-metric}) 
\[
g:\cX(M)\times\cX(M)\rightarrow \cC^\infty(M).
\]
Again the value $g(X,Y)$ at a point $p$ depends only on $X_p$ and $Y_p$, so $g$ restricts to an $\RR$-bilinear map $g_p:T_pM\times T_pM\rightarrow \RR$. The nondegeneracy of $g$ means that $g_p$ is nondegenerate for any $p\in M$. 

Prove that for any pseudo-Riemannian manifold $(M,g)$ there is a unique affine connection $\nabla$ such that:
\begin{romanenumerate}
\item $\nabla$ is symmetric,
\item $Xg(Y,Z)=g(\nabla_XY,Z)+g(Y,\nabla_XZ)$ for any $X,Y,Z\in\cX(M)$.
\end{romanenumerate}
This connection is called the \emph{Levi-Civita connection}. For this connection, the parallel transport defines isometries among the tangent spaces at different points.

\noindent{\small [Hint: Permute cyclically $X,Y,Z$ in condition (ii) and combine the resulting equations, using (i), to get that $g(\nabla_XY,Z)$ is uniquely determined. Now use that $g$ is nondegenerate.]}
\end{exercise}

\bigskip

\section{Lie groups and Lie algebras}\label{se:LGLA}

A \emph{Lie group} is a manifold $G$ which is also an abstract group such that the multiplication map
\[
\begin{split}
G\times G&\longrightarrow G\\
(g_1,g_2)&\mapsto g_1g_2,
\end{split}
\]
and the inversion map
\[
\begin{split}
G&\longrightarrow G\\
g& \mapsto g^{-1}
\end{split}
\]
are smooth maps.

\begin{exercise} Identify the unit circle $\bS^1$ with the set of norm one complex numbers. This set is a group under complex multiplication. Check that in this way $\bS^1$ becomes a Lie group.
\end{exercise}

\begin{remark} Given an element $g$ in a Lie group $G$, the left multiplication by $g$ gives a map $L_g:G\rightarrow G$ which is a diffeomorphism (with inverse $L_{g^{-1}}$). The same applies to the right multiplication $R_g$.
\end{remark}

The \emph{Lie algebra} of a Lie group $G$ is the real vector subspace of ``left-invariant vector fields'':
\[
\mathcal{Lie}(G)\bydef \{X\in\cX(G): (L_g)_*(X)=X\ \forall g\in G\}.
\]
By Exercise \ref{exe:Phi_related}, $\frg=\mathcal{Lie}(G)$ is a Lie subalgebra of $\cX(G)$. Moreover, the natural map
\[
\begin{split}
\frg&\longrightarrow T_eG\\
X&\mapsto \ X_e,
\end{split}
\]
where $e$ denotes the neutral element of the group structure, turns out to be a linear isomorphism (check this!), as any left-invariant vector field is determined uniquely by its value at any given point. Hence the dimension of $\frg$ as a vector space coincides with the dimension of $G$ as a manifold.

\begin{example} The general linear group $\GL_n(\RR)$ consists of the $n\times n$ regular matrices over $\RR$. This is the inverse image of the open set $\RR\setminus\{0\}$ under the smooth map (actually it is a polynomial map) given by the determinant $\det: \Mat_n(\RR)\rightarrow \RR$. Thus $\GL_n(\RR)$ is an open set in the euclidean space $\Mat_n(\RR)\simeq \RR^{n^2}$ and hence it is a manifold too. The multiplication is given by a polynomial map and the inversion by a rational map, and hence they are both smooth.

Since $\GL_n(\RR)$ is open in $\RR^{n^2}$ there is the global chart with standard coordinates $\{x_{ij}: 1\leq i,j\leq n\}$, where $x_{ij}(A)$ denotes the $(i,j)$-entry of the matrix $A$. Then the tangent space at any point $A\in G=\GL_n(\RR)$ can be identified with the vector space of $n\times n$-matrices $\Mat_n(\RR)$ by means of the map:
\begin{equation}\label{eq:TAG}
\begin{split}
T_A G&\longrightarrow \Mat_n(\RR)\\
v\ &\mapsto\ \Bigl(v(x_{ij})\Bigr)_{i,j=1}^n.
\end{split}
\end{equation}
(Note that $v=\sum_{i,j}v(x_{ij})\left.\frac{\partial\ }{\partial x_{ij}}\right|_{A}$.)

Let $\frg$ be the Lie algebra of $G=\GL_n(\RR)$, and let $X\in \frg$. For any $A=\bigl(a_{ij}\bigr)\in G$, the left-invariance of $X$ gives:
\[
\begin{split}
X_A(x_{ij})&=(L_A)_*(X_{I_n})(x_{ij})=X_{I_n}(x_{ij}\circ L_A)\\
  &=X_{I_n}\left(\sum_{k=1}^n a_{ik}x_{kj}\right)\quad\textrm{(the $(i,j)$-entry of $AB$ is $\sum_{k=1}^n a_{ik}b_{kj}$)}\\
  &=\sum_{k=1}^n a_{ik}X_{I_n}(x_{kj})=\sum_{k=1}^nX_{I_n}(x_{kj})x_{ik}(A).
\end{split}
\]
Therefore, we obtain
\begin{equation}\label{eq:X_left_invariant}
Xx_{ij}=\sum_{k=1}^nX_{I_n}\bigl(x_{kj}\bigr)x_{ik}.
\end{equation}
(Here $I_n$ denotes the identity matrix, which is the neutral element of $G$.)

Consider now two elements $X,Y\in\frg$. $X$ and $Y$ are determined by $X_{I_n}$ and $Y_{I_n}$ respectively, and these can be identified as in \eqref{eq:TAG} with the matrices $A\bydef \Bigl(X_{I_n}(x_{ij})\Bigr)$ and $B\bydef \Bigl(Y_{I_n}(x_{ij})\Bigr)$.
Then we get
\[
\begin{split}
[X,Y]_{I_n}(x_{ij})&=X_{I_n}(Yx_{ij})-Y_{I_n}(Xx_{ij})\\
  &=X_{I_n}\bigl(\sum_{k=1}^nb_{kj}x_{ik}\bigr)
    -Y_{I_n}\bigl(\sum_{k=1}^n a_{kj}x_{ik}\bigr)\qquad\textrm{(because of \eqref{eq:X_left_invariant})}\\
     &=\sum_{k=1}^nb_{kj}a_{ik}-\sum_{k=1}^na_{kj}b_{ik}=x_{ij}\bigl([A,B]\bigr).
\end{split}
\]
Denote by $\frgl_n(\RR)$ the Lie algebra defined on the vector space $\Mat_n(\RR)$ with the usual Lie bracket $[A,B]=AB-BA$. Then the above computation shows that the linear map
\[
\begin{split}
\frg&\longrightarrow \frgl_n(\RR)\\
X&\mapsto \Bigl(X_{I_n}(x_{ij})\Bigr)_{i,j=1}^n,
\end{split}
\]
is a Lie algebra isomorphism.

Moreover, for any $A\in\frgl_n(\RR)$ the map
\[
\begin{split}
\gamma_A:\RR&\longrightarrow G\\
 t&\mapsto \exp(tA)\bydef \sum_{n=0}^\infty \frac{t^nA^n}{n!},
\end{split}
\]
is a smooth group homomorphism. Besides, $\dot\gamma(0)=A$ (with the identification above $T_{I_n}G\simeq \frgl_n(\RR)$). We conclude that the smooth map
\[
\begin{split}
\Phi_A:\RR\times G&\longrightarrow G\\
(t,B)&\mapsto B\exp(tA)
\end{split}
\]
is the global flow of the left-invariant vector field $X$ with $X_{I_n}=A$.
\end{example}

We may substitute $\GL_n(\RR)$ by $\GL(V)$ (the group of linear automorphisms of $V$) for a finite-dimensional real vector space $V$, and $\frgl_n(\RR)$ by $\frgl(V)$ (the Lie algebra of linear endomorphisms of $V$ with the natural Lie bracket) in the Example above.
Once we fix a basis of $V$ we get  isomorphisms  $\GL(V)\cong\GL_n(\RR)$ (of Lie groups) and $\frgl(V)\cong\frgl_n(\RR)$ (of Lie algebras).

\medskip

In general, left-invariant vector fields on a Lie group $G$ are always complete and there is a smooth map from the Lie algebra $\frg$ to $G$, called the \emph{exponential}:
\begin{equation}\label{eq:exp}
\exp:\frg\rightarrow G,
\end{equation}
which restricts to a diffeomorphism of a neighborhood of $0\in\frg$ ($\frg\simeq\RR^n$) onto an open neighborhood of $e\in G$ such that for any $X\in\frg$, its global flow is given by the map
\[
\begin{split}
\Phi:\RR\times G&\longrightarrow G\\
(t,g)\ &\mapsto g\exp(tX).
\end{split}
\]
Any open neighborhood of $e\in G$ generates, as an abstract group, the connected component of $e$, which is a normal subgroup. In particular, the subgroup generated by $\exp(\frg)$ is the connected component $G_0$ of $e$.

\begin{remark}\label{re:PhitR} 
In this situation, for any $t\in\RR$, $\Phi_t$ is the diffeomorphism $R_{\exp(tX)}$ (right multiplication by $\exp(tX)$). Note that for any $h\in G$, the left invariance of any $X\in\frg$ gives $(L_h)_*(X)=X$, so for any $g,h\in G$,
\[
\Bigl((L_h)_*(X)\Bigr)_{hg}f=X_g(f\circ L_h)
=\left.\frac{d\ }{dt}\right|_{t=0}(f\circ L_h)(g\exp(tX)),
\]
and this certainly agrees with
\[
X_{hg}f=\left.\frac{d\ }{dt}\right|_{t=0}f(hg\exp(tX)).
\]
Also, condition \eqref{eq:Phi_t1t2} implies that the smooth map $\RR\rightarrow G$, $t\mapsto \exp(tX)$ is a group homomorphism (called a \emph{one-parameter subgroup}).
\end{remark}

We may consider too the real subspace of ``right-invariant'' vector fields:
\[
\frg_\textrm{right}\bydef \{X\in\cX(G): \bigl(R_g\bigr)_*(X)=X\ \forall g\in G\}.
\]
Again, the natural map $\frg_\textrm{right}\rightarrow T_eG$, $X\mapsto X_e$, is a linear isomorphism.

\begin{proposition}\label{pr:XXhat}
 Let $G$ be a Lie group with Lie algebra $\frg$. Consider the linear isomorphisms
\begin{align*}
\varphi_\textrm{left}:\frg&\longrightarrow T_eG\quad&
\quad\varphi_\textrm{right}: \frg_\textrm{right}&\longrightarrow T_eG\\
X&\mapsto\ X_e&\hat X\ &\mapsto\ \hat X_e.
\end{align*}
Then the map 
\[
\begin{split}
\frg&\longrightarrow\frg_\textrm{right}\\
X&\mapsto -\varphi_\textrm{right}^{-1}\circ\varphi_\textrm{left}(X)
\end{split}
\]
is an isomorphism of Lie algebras.

In other words, given $v\in T_eG$, let $X\in\frg$ (respectively $\hat X\in\frg_\textrm{right}$) be the left (respectively right) invariant vector field with $X_e=v$ (respectively $\hat X_e=v$). Then the map $X\mapsto -\hat X$ is a Lie algebra isomorphism.
\end{proposition}
\begin{proof}
The smooth map 
\[
\begin{split}
\Phi^\textrm{right}:\RR\times G&\longrightarrow G\\
 (t,g)&\mapsto \exp(tX)g.
 \end{split}
\]
is the global flow of $\hat X$, because for any $f\in\cC^\infty(G)$:
\[
\begin{split}
\hat X_g(f)&=\Bigl((R_g)_*(\hat X_e)\Bigr)(f)=\hat X_e(f\circ R_g)\\
 &=v(f\circ R_g)=X_e(f\circ R_g)\\
 &=\left.\frac{d\ }{dt}\right|_{t=0}(f\circ R_g)(\exp(tX))
  = \left.\frac{d\ }{dt}\right|_{t=0}f(\exp(tX)g).
\end{split}
\]
Now consider the inversion map $\iota:G\rightarrow G$, $g\mapsto g^{-1}$. It is a diffeomorphism, so $\iota_*:\cX(G)\rightarrow \cX(G)$ is an isomorphism of Lie algebras.

For $g\in G$, $X$ and $\hat X$ as in the Proposition, and $f\in\cC^\infty(G)$ we get:
\[
\begin{split}
\bigl(\iota_*X)_{g^{-1}}f&=X_g(f\circ \iota)
   =\left.\frac{d\ }{dt}\right|_{t=0}(f\circ\iota)(g\exp(tX))\\
   &=\left.\frac{d\ }{dt}\right|_{t=0}f\bigl((g\exp(tX))^{-1}\bigr)\\
   &=\left.\frac{d\ }{dt}\right|_{t=0}f(\exp(-tX)g^{-1})=-\hat X_{g^{-1}}(f).
\end{split}
\]
We conclude that $\iota_*(X)=-\hat X$, so $X$ and $-\hat X$ are $\iota$-related. In particular, this shows that $\iota_*$ takes left invariant vector fields to right invariant vector fields, and hence it gives an isomorphism from $\frg$ to $\frg_\textrm{right}$.
\end{proof}

\medskip

For any element $g$ in a Lie group $G$, conjugation by $g$ gives a diffeomorphism:
\[
\begin{split}
\iota_g:G&\longrightarrow G\\
  x&\mapsto gxg^{-1},
\end{split}
\]
and this induces a Lie algebra isomorphism $\Ad_g\bydef (\iota_g)_*:\cX(G)\rightarrow \cX(G)$ that preserves $\frg$. Indeed, for $X\in\frg$ and $h\in G$
\[
\begin{split}
(L_h)_*(\Ad_g X)&=(L_h)_*\circ(\iota_g)_*(X)=
      (L_h)_*\circ(L_g)_*\circ(R_{g^{-1}})_*(X)\\
    &=(R_{g^{-1}})_*\circ (L_h)_*\circ (L_g)_*(X)\\
    &\qquad\textrm{(because $L_xR_y=R_yL_x$ for any $x,y\in G$)}\\
    &=(R_{g^{-1}})_*X\quad\textrm{(as $X$ is left invariant)}\\
    &=(R_{g^{-1}})_*\circ(L_g)_*(X)=\Ad_g X.
\end{split}
\]
We obtain in this way the so called \emph{adjoint representation}:
\[
\begin{split}
\Ad: G&\longrightarrow \GL(\frg)\\
  g&\mapsto \Ad_g\, ,
\end{split}
\]
which is a homomorphism of Lie groups (i.e., a smooth group homomorphism).

\medskip

Given a homomorphism of Lie groups $\varphi:G\rightarrow H$ and a left invariant vector field $X\in\frg$ (the Lie algebra of $G$), the map $t\mapsto \varphi(\exp(tX))$ gives a Lie group homomorphism $\RR\rightarrow H$ and hence the map
\[
\begin{split}
\Phi:\RR\times H&\longrightarrow H\\
(t,h)&\mapsto h\varphi(\exp(tX))
\end{split}
\]
is the global flow of a left invariant vector field $Y\in\frh$ (the Lie algebra of $H$), which is $\varphi$-related to $X$. This induces a Lie algebra homomorphism $\varphi_*:\frg\rightarrow\frh$ such that the diagram 
\begin{equation}\label{eq:varphi_exp}
\begin{gathered}
\xymatrix{
\frg \ar[r]^{\varphi_*} \ar[d]_\exp &\frh \ar[d]^\exp \\
G \ar[r]^\varphi  &H}
\end{gathered}
\end{equation}
is commutative.

In particular, we obtain a commutative diagram
\begin{equation}\label{eq:Ad_exp}
\begin{gathered}
\xymatrix{
\frg \ar[r]^{\Ad_*\ } \ar[d]_\exp &\frgl(\frg) \ar[d]^\exp \\
G \ar[r]^{\Ad\quad}  &\GL(\frg)}
\end{gathered}
\end{equation}
where, on the right $\exp$ denotes the standard exponential of linear endomorphisms.

For $X,Y\in\frg$ we get:
\[
\begin{split}
\Ad_{\exp(tX)}Y&=\bigl(\iota_{\exp(tX)}\bigr)_*Y\\
&=\bigl(R_{\exp(-tX)}\bigr)_*\circ\bigl(L_{\exp(tX)}\bigr)_*(Y)\\
&=\bigl(R_{\exp(-tX)}\bigr)_*(Y)\quad\textrm{(as $Y$ is left invariant).}
\end{split}
\]
Since $\Phi_t=R_{\exp(tX)}$ is the global flow of $X$ (Remark \ref{re:PhitR}), we obtain from \eqref{eq:XYPhi} the equation
\begin{equation}\label{eq:Rexp-tXY}
\left.\frac{d\ }{dt}\right|_{t=0}\bigl(R_{\exp(-tX)}\bigr)_*(Y)=[X,Y].
\end{equation}
Therefore the map $t\mapsto \Ad_{\exp(tX)}Y$ is the one-parameter group in $\frg\simeq T_eG$ whose derivative at $0$ is $[X,Y]$. Alternatively, $t\mapsto \Ad_{\exp(tX)}$ is the one-parameter group in the Lie group $\GL(\frg)$ whose derivative at $0$ is the linear endomorphism $\ad_X:Y\mapsto [X,Y]$ ($\ad_X\in\frgl(\frg)$, the Lie algebra of $\GL(\frg)$).

We conclude that
\begin{equation}\label{eq:AdtXad}
\Ad_{\exp(tX)}=\exp(t\ad_X)
\end{equation}
for any $X\in\frg$ and $t\in\RR$. Hence we obtain $(\Ad)_*=\ad:\frg\rightarrow\frgl(\frg)$, and the diagram in \eqref{eq:Ad_exp} becomes the following commutative diagram:
\begin{equation}\label{eq:Adad}
\begin{gathered}
\xymatrix{
\frg \ar[r]^{\ad\ } \ar[d]_\exp &\frgl(\frg) \ar[d]^\exp \\
G \ar[r]^{\Ad\quad}  &\GL(\frg)}
\end{gathered}
\end{equation}

\begin{remark}
On the left hand side of \eqref{eq:Adad} we have an arbitrary Lie group and its Lie algebra, while on the right hand side we have a very concrete Lie group: the general linear Lie group on a vector space, and its Lie algebra.
\end{remark}

\bigskip

\section{Invariant affine connections on homogeneous spaces}\label{se:Homog_spaces}

\subsection{Homogeneous spaces}

Let $G$ be a Lie group with Lie algebra $\frg$, $M$ a manifold and let
\[
\begin{split}
\tau:G\times M&\longrightarrow M\\
 (g,p)\ &\mapsto \tau(g,p)=g\cdot p
\end{split}\]
be a smooth action. That is, $\tau$ is both a smooth map and a group action.

If the group  action $\tau$ is transitive (i.e., for any $p,q\in M$ there is an element $g\in G$ such that $g\cdot p=q$), then $M$ is said to be a \emph{homogeneous space}.
(Formally, we should consider homogeneous spaces as triples $(M,G,\tau)$.)

\begin{example}
The $n$-dimensional sphere 
\[
\bS^n\bydef \{(x_1,\ldots,x_{n+1})\in\RR^{n+1}:\sum_{i=1}^{n+1}x_i^2=1\}
\] 
is a homogeneous space relative to the natural action of the special orthogonal group: $SO(n+1)\times\bS^N\rightarrow \bS^n$.

On the other hand, $\bS^{2n-1}$ can be identified with the set $\{(z_1,\ldots,z_n)\in\CC^n:\sum_{i=1}^n|z_i|^2=1\}$, and hence it is a homogeneous space relative to the natural action of the special unitary group: $SU(n)\times \bS^{2n-1}\rightarrow \bS^{2n-1}$.
\end{example}

Given a homogeneous space $M$, fix an element $p\in M$ and consider the corresponding \emph{isotropy subgroup}:
\[
H\bydef \{g\in G: g\cdot p=p\}.
\]
This is a closed subgroup of $G$ and hence (see e.g. \cite[Theorem 5.3.2]{Conlon}) it is a Lie group too, with Lie algebra
\[
\frh\bydef \{X\in\frg : \exp(tX)\in H\ \forall t\in\RR\}.
\]
(Recall the map $\exp:\frg\rightarrow G$ in equation \eqref{eq:exp}.) The Lie algebra $\frh$ is then a subalgebra of $\frg$.

Moreover, the set of left cosets $G/H$ is a manifold with a suitable atlas such that the bijection
\[
\begin{split}
G/H&\longrightarrow M\\
gH&\mapsto g\cdot p
\end{split}
\]
is a diffeomorphism (see e.g. \cite[Proposition 5.4.12]{Conlon}).

For any $X\in\frg$ we can define a global flow:
\[
\begin{split}
\Phi:\RR\times M&\longrightarrow M\\
(t,m)\ &\mapsto \exp(tX)\cdot m.
\end{split}
\]
The associated vector field in $\cX(M)$ will be denoted by $X^+$. Hence we have
\[
X_m^+(f)=\left.\frac{d\ }{dt}\right|_{t=0}f\bigl(\exp(tX)\cdot m\bigr)
\]
for any $m\in M$ and $f\in\cC^\infty(M)$. The next result summarizes some of the main properties of these vector fields.

\begin{proposition}\label{pr:properties_X+}
Let $M$ be a homogenous space of the Lie group $G$ as above. Then the following properties hold:
\begin{romanenumerate}
\item For any $q\in M$, let $\pi_q:G\rightarrow M$ be the smooth map $g\mapsto g\cdot q$. Then for any $X\in\frg$ and $f\in\cC^\infty(M)$
\[
X_q^+(f)=X_e(f\circ\pi_q).
\]
In other words, $X^+_q=(\pi_q)_*(X_e)$.

\item For any $X\in\frg$, let $\hat X$ be the right invariant vector field on $G$ with $X_e=\hat X_e$ (see Proposition \ref{pr:XXhat}). Then $\hat X$ and $X^+$ are $\pi_q$-related for any $q\in M$. Moreover, $X^+$ is determined by this property.

\item The map
\[
\begin{split}
\frg&\longrightarrow \cX(M)\\
X&\mapsto\ -X^+
\end{split}
\]
is a Lie algebra homomorphism.

\item For any $g\in G$, denote by $\tau(g)$ the diffeomorphism $M\rightarrow M$, $q\mapsto g\cdot q$. Then for any $g\in G$ and $X\in\frg$ we have
\[
\tau(g)_*(X^+)=\bigl(\Ad_gX)^+.
\]

\item Fix, as before, a point $p\in M$, then $\pi\bydef \pi_p:G\rightarrow M$ induces a linear surjection
\[
\begin{split}
\pi_*:\frg\,(\simeq T_e(G))&\longrightarrow T_pM\\
X\quad&\mapsto \left.\frac{d\ }{dt}\right|_{t=0}\pi(\exp(tX))=X^+_p,
\end{split}
\]
with $\ker\pi_*=\frh$ (the Lie algebra of the isotropy subgroup $H$ at $p$).
\end{romanenumerate}
\end{proposition}
\begin{proof}
For (i), we proceed as follows:
\[
X_q^+(f)=\left.\frac{d\ }{dt}\right|_{t=0}f\bigl(\exp(tX)\cdot q\bigr)=
\left.\frac{d\ }{dt}\right|_{t=0}f\bigl(\pi_q(\exp(tX))\bigr)=X_e(f\circ\pi_q).
\]

For (ii), take any $f\in\cC^\infty(M)$ and $g\in G$ to obtain
\[
\begin{split}
\bigl(X^+f\bigr)\circ\pi_q(g)&=\bigl(X^+f\bigr)(g\cdot q)=X^+_{g\cdot q}(f)\\
&=X_e(f\circ\pi_{g\cdot q})=X_e(f\circ\pi_q\circ R_g)\\
&=\hat X_g(f\circ\pi_q)=\bigl(\hat X(f\circ\pi_q)\bigr)(g).
\end{split}\]
Hence $(X^+f)\circ\pi_q=\hat X(f\circ\pi_q)$, as required.

By Exercise \ref{exe:Phi_related}, for any $X,Y\in\frg$, the vector fields $[\hat X,\hat Y]$ in $G$ and $[X^+,Y^+]$ in $M$ are $\pi_q$-related for any $q\in M$. But $[\hat X,\hat Y]=-\widehat{[X,Y]}$ by Proposition \ref{pr:XXhat}, so $-\widehat{[X,Y]}$ and $[X^+,Y^+]$ are $\pi_q$-related for any $q\in M$, and this shows that $[X^+,Y^+]=-[X,Y]^+$, thus proving (iii).

Now, for any $q\in M$ and $f\in\cC^\infty(M)$ we obtain:
\[
\begin{split}
\bigl(\tau(g)_*(X^+)\bigr)_{g\cdot q}(f)
 &=X^+_q(f\circ\tau(g))=X_e(f\circ\tau(g)\circ\pi_q)\\
 &=X_e(f\circ\pi_{g\cdot q}\circ \iota_g)\\
 &\qquad\textrm{(as $\pi_{g\cdot q}\circ\iota_g(h)=\pi_{g\cdot q}(ghg^{-1})=(gh)\cdot q
   =\tau(g)\circ\pi_q(h)$)}\\
 &=\bigl(\Ad_g X\bigr)_e(f\circ\pi_{g\cdot q})=\bigl(\Ad_g X\bigr)^+_{g\cdot q}(f),
\end{split}
\]
and this proves (iv).

Finally, 
\[
\begin{split}
\ker\pi_*&=\left\{X\in\frg : \left.\frac{d\ }{dt}\right|_{t=0}\pi\bigl(\exp(tX)\bigr)=0\right\}\\
 &=\left\{X\in\frg : \left.\frac{d\ }{dt}\right|_{t=0}\exp(tX).p=0\right\}
\end{split}
\]
contains $\frh$. But $\pi_*$ is  surjective, and hence $\ker\pi_*=\frh$ by dimension count.
\end{proof}

Therefore, in the conditions of the Proposition above, the map 
\begin{equation}\label{eq:ghTpM}
\begin{split}
\phi:\frg/\frh&\longrightarrow T_pM\\
X+\frh&\mapsto\quad X^+_p
\end{split}
\end{equation}
is a linear isomorphism that allows us to identify $T_pM$ with $\frg/\frh$. Under this identification, for any $h\in H$, the map $\tau(h)_*:T_pM\rightarrow T_pM$ corresponds, because of Proposition \ref{pr:properties_X+}(iv), to the map
\begin{equation}\label{eq:jh}
\begin{split}
j(h):\frg/\frh&\longrightarrow \frg/\frh\\
X+\frh&\mapsto \bigl(\Ad_h X\bigr)+\frh.
\end{split}
\end{equation}
That is, we get a commutative diagram
\begin{equation}\label{eq:jhtauh}
\begin{gathered}
\xymatrix{
\frg/\frh \ar[r]^{j(h)} \ar[d]_\phi &\frg/\frh \ar[d]^\phi\\
T_pM \ar[r]_{\tau(h)_*} &T_pM}
\end{gathered}
\end{equation}

The map $j:H\rightarrow GL(\frg/\frh)$, $h\mapsto j(h)$, is the natural representation of $H$ on $\frg/\frh$ (induced by the adjoint representation).

\medskip

In what follows, when we refer to a homogeneous space $M\simeq G/H$, we mean that $G$ is a Lie group acting smoothly and transitively on $M$, and that a point $p\in M$ has been fixed with $H$ as isotropy subgroup.

\begin{remark}\label{re:X+_span}
Proposition \ref{pr:properties_X+}(v) shows that for any $q\in M$, the subspace $\{X^+_q:X\in\frg\}$ fills the whole $T_qM$. In other words, the space $\{X^+: X\in\frg\}$ spans $\cX(M)$ as a $\cC^\infty(M)$-module.
\end{remark}

\medskip

\subsection{Invariant affine connections on homogeneous spaces}

\begin{definition}
Let $M\simeq G/H$ be a homogeneous space. An affine connection $\nabla$ on $M$ is said to be \emph{invariant} if for any $X,Y\in\cX(M)$ and $g\in G$,
\[
\tau(g)_*\Bigl(\nabla_XY\Bigr)=\nabla_{\tau(g)_*(X)}\tau(g)_*(Y).
\]
\end{definition}

\begin{proposition}\label{pr:Ltau}
Let $\nabla$ be an invariant affine connection on a homogeneous space $M\simeq G/H$. Then for any  $X\in\frg$ (the Lie algebra of $G$) and any $g\in G$ we have:
\[
\bL_{(\Ad_g X)^+}=\tau(g)_*\circ \bL_{X^+}\circ\tau(g^{-1})_*.
\]
\end{proposition}
\begin{proof}
For any $X\in\frg$, recall from Proposition \ref{pr:properties_X+}(iv) that $(\Ad_g X)^+=\tau(g)_*(X^+)$, so it is enough to prove that $\bL_{\tau(g)_*U}\circ\tau(g)_*=\tau(g)_*\circ \bL_U$ for any $U\in\cX(M)$, and this is a straightforward consequence of the invariance of $\nabla$.
\end{proof}

Therefore, the invariant affine connection $\nabla$, which is determined by the Nomizu operators $\bL_{X^+}$ for $X\in\frg$ because of Remark \ref{re:X+_span}, is actually determined by the endomorphisms $\left.\bL_{X^+}\right|_p$, because for any $q\in M$, there is an element $g\in G$ such that $q=g\cdot p$, and for $Y\in\frg$ we have
\[
\begin{split}
\bL_{X^+}(Y^+_{g\cdot p})&=
  \bL_{X^+}\Bigl(\tau(g)_*(\Ad_{g^{-1}}Y)^+_p\Bigr)\quad\textrm{(Proposition \ref{pr:properties_X+}(iv))}\\
  &=\tau(g)_*\bL_{(\Ad_{g^{-1}} X)^+}\Bigl((\Ad_{g^{-1}} Y)^+_p\Bigr).
\end{split}
\]
If we identify $T_pM$ with $\frg/\frh$ via the map $\phi$ in \eqref{eq:ghTpM} we may think of $\left.\bL_{X^+}\right|_p$ as a linear endomorphism of $\frg/\frh$.

Hence $\nabla$ is determined by the linear map, also denoted by $\bL$:
\begin{equation}\label{eq:L}
\begin{split}
\bL:\frg&\longrightarrow \End_\RR(\frg/\frh)\\
X&\mapsto \begin{aligned}[t]
\bL_X:\frg/\frh&\longrightarrow \frg/\frh\\
Y+\frh&\mapsto \phi^{-1}\bigl(\bL_{X^+}(Y^+_p)\bigr).\end{aligned}
\end{split}
\end{equation}

\begin{proposition}\label{pr:properties_ab}
Let $\nabla$ be an invariant affine connection on the homogeneous space $M\simeq G/H$. Then the linear map $\bL$ just defined satisfies the following properties:
\begin{enumerate}
\item[(a)] $\bL_{\Ad_hX}=j(h)\circ \bL_X\circ j(h^{-1})$ for any $X\in\frg$ and $h\in H$.
\item[(b)] $\bL_X(Y+\frh)=[X,Y]+\frh$ for any $X\in\frh$ and $Y\in\frg$.
\end{enumerate}
\end{proposition}

In other words, $\bL:\frg\rightarrow \End_{\RR}(\frg/\frh)$ is a homomorphism of $H$-modules that extends the natural (adjoint) representation of $\frh$ on $\frg/\frh$.

Actually, (a) is a direct consequence of Proposition \ref{pr:Ltau} and the commutativity of \eqref{eq:jhtauh}, and (b) follows because for $X\in\frh$, $X^+_p=0$, so $\bigl(\nabla_{X^+}Y^+\bigr)_p=0$, and hence $\bigl(\bL_{X^+}(Y^+)\bigr)_p=\bigl(\nabla_{X^+}Y^+\bigr)_p-[X^+,Y^+]_p=[X,Y]^+_p$.

Conversely, given an $\RR$-linear map $\bL:\frg\rightarrow \End_\RR(\frg/\frh)$ satisfying the properties (a) and (b) above, we may define, for $X,Y\in\frg$ and $g\in G$:
\begin{itemize}
\item $\nabla_{X^+_p}Y^+\bydef\phi\bigl(\bL_X(Y+\frh)\bigr)-[X,Y]^+_p\in T_pM$ (well defined by property (b)),
\item $\nabla_{X^+_{g\cdot p}}Y^+\bydef\tau(g)_*\left(\nabla_{(\Ad_{g^{-1}} X)_p^+}(\Ad_{g^{-1}}Y)^+\right)$ (well defined by property (a)),
\end{itemize}
and this defines an invariant affine connection on $M$. 

We summarize the above arguments in the next result, which goes back to Vinberg \cite[Theorem 2]{Vinberg}.

\begin{theorem}\label{th:nablaL}
The invariant affine connections on a homogeneous space $M\simeq G/H$ are in bijection with the $\RR$-linear maps $\bL:\frg\rightarrow\End_\RR(\frg/\frh)$ satisfying the properties \textrm{(a)} and \textrm{(b)} in Proposition \ref{pr:properties_ab}.
\end{theorem}

\begin{remark} Condition (a) above means that the map $\bL$ in the Theorem is a homomorphisms of $H$-modules, and this implies that it is a homomorphism of $\frh$-modules:
\[
\bL_{[U,X]}=[\tilde\jmath(U),\bL_X]
\]
for any $U\in\frh$ and $X\in\frg$, where $\tilde\jmath(U)$ is the endomorphism of $\frg/\frh$ such that $\tilde\jmath(U)(Y+\frh)=[U,Y]+\frh$ for any $Y\in\frg$. Indeed, property (a) gives
\[
\bL_{\Ad_{\exp(tU)}X}\circ j(\exp(tU))=j(\exp(tU))\circ \bL_X
\]
in $\End_\RR(\frg/\frh)$ for any $t\in\RR$. Taking the derivative at $t=0$ we get $\bL_{[U,X]}+\bL_X\circ\tilde\jmath(U)=\tilde\jmath(U)\circ \bL_X$, as required.
\end{remark}

\medskip

The torsion and curvature tensors of an invariant affine connection on a homogeneous space $M\simeq G/H$ can now be expressed in terms of the linear map $\bL$ in \eqref{eq:L}, because, by invariance, it is enough to compute them at the point $p$.

\begin{description}
\item[Torsion] For $X,Y\in\frg$ we have
\[
\begin{split}
T(X^+,Y^+)&\bydef \nabla_{X^+}Y^+-\nabla_{Y^+}X^+-[X^+,Y^+]\\
 &=\bL_{X^+}(Y^+)-\bL_{Y^+}(X^+)+[X^+,Y^+]\\
 &=\bL_{X^+}(Y^+)-\bL_{Y^+}(X^+)-[X,Y]^+\quad\textrm{(by Proposition \ref{pr:properties_X+}(iii)),}
\end{split}
\]
so at the point $p$ we get
\[
T(X^+,Y^+)_p=\phi\left(\bL_X(Y+\frh)-\bL_Y(X+\frh)-\bigl([X,Y]+\frh\bigr)\right),
\]
which corresponds, through the map $\phi$ in \eqref{eq:ghTpM} to
\begin{equation}\label{eq:TXhYh}
T(X+\frh,Y+\frh)\bydef \bL_X(Y+\frh)-\bL_Y(X+\frh)-\bigl([X,Y]+\frh\bigr).
\end{equation}

\item[Curvature] Again, for $X,Y\in\frg$ we have
\[
\begin{split}
R(X^+,Y^+)&\bydef [\nabla_{X^+},\nabla_{Y^+}]-\nabla_{[X^+,Y^+]}\\
 &=[\bL_{X^+}+\ad_{X^+},\bL_{Y^+}+\ad_{Y^+}]-\nabla_{[X^+,Y^+]}\\
 &=[\bL_{X^+},\bL_{Y^+}]+[\ad_{X^+},\bL_{Y^+}]+[\bL_{X^+},\ad_{Y^+}]+\ad_{[X^+,Y^+]}-\nabla_{[X^+,Y^+]}\\
 &=[\bL_{X^+},\bL_{Y^+}]+[\ad_{X^+},\bL_{Y^+}]+[\bL_{X^+},\ad_{Y^+}]-\bL_{[X^+,Y^+]}.
\end{split}
\]
\end{description}

Now, given a manifold $M$ with an affine connection, an \emph{affine transformation} is a diffeomorphism $\varphi$ such that $\varphi_*(\nabla_XY)=\nabla_{\varphi_*(X)}\varphi_*(Y)$ for any vector fiels $X,Y\in\cX(M)$, and a vector field $X\in\cX(M)$ is called an \emph{infinitesimal affine transformation} if its local flow $\Phi_t$ is an affine transformation for $-\epsilon<t<\epsilon$ for some $\epsilon>0$. 

If $X\in\cX(M)$ is an infinitesimal affine transformation, for any $Y,Z\in\cX(M)$ we get:
\begin{equation}\label{eq:infinitesima_affine}
\begin{split}
[X,\nabla_YZ]&=\left.\frac{d\ }{dt}\right|_{t=0}(\Phi_{-t})_*(\nabla_YZ)\ \textrm{(because of \eqref{eq:XYPhi})}\\
 &=\left.\frac{d\ }{dt}\right|_{t=0}\left(\nabla_{(\Phi_{-t})_*(Y)}(\Phi_{-t})_*(Z)\right)\\
 &=\nabla_{[X,Y]}Z+\nabla_Y[X,Z].
\end{split}
\end{equation}
The last equality above is proved as follows. Locally (i.e., in local coordinates) $Y(t,q)\bydef(\Phi_{-t})_*Y=\sum_i f_i(t,q)\frac{\partial\ }{\partial x_i}$ and $Z(t,q)\bydef(\Phi_{-t})_*Z=\sum_i g_i(t,q)\frac{\partial\ }{\partial x_i}$. Hence
\[
\begin{split}
\nabla_{Y(t,q)}Z(t,q)
=\sum_{i,j}f_i(t,q)\frac{\partial g_j(t,q)}{\partial x_i}\frac{\partial\ }{\partial x_j}
\,+\,\sum_{i,j}f_i(t,q)g_j(t,q)\Gamma_{i,j}(q)\frac{\partial\ }{\partial x_k}.
\end{split}
\]
(Note that the Christoffel symbols do not depend on $t$.) Now take the derivative with respect to $t$ at $t=0$ using Leibniz's rule to get the required equality.

Returning to our homogeneous space $M\simeq G/H$ with an invariant affine connection $\nabla$, for any $X\in\frg$, $X^+$ is an infinitesimal affine transformation, because its global flow is $\tau(\exp(tX))$ and $\nabla$ is invariant. Hence, by \eqref{eq:infinitesima_affine}, we get for any $X,Y\in\frg$:
\[
\begin{split}
[\ad_{X^+},\bL_{Y^+}]&=[\ad_{X^+},\nabla_{Y^+}+\ad_{Y^+}]
 =\nabla_{[X^+,Y^+]}+\ad_{[X^+,Y^+]}\\
 &=\bL_{[X^+,Y^+]}=-\bL_{[X,Y]^+},
\end{split}
\]
so the formula above for $R(X^+,Y^+)$ becomes:
\[
\begin{split}
R(X^+,Y^+)&=[\bL_{X^+},\bL_{Y^+}]+\bL_{[X^+,Y^+]}-\bL_{[Y^+,X^+]}-\bL_{[X^+,Y^+]}\\
 &=[\bL_{X^+},\bL_{Y^+}]+\bL_{[X^+,Y^+]}=[\bL_{X^+},\bL_{Y^+}]-\bL_{[X,Y]^+},
\end{split}
\]
which at the point $p\in M$ corresponds via $\phi$ in \eqref{eq:ghTpM} to
\begin{equation}\label{eq:RXhYh}
R(X+\frh,Y+\frh)=[\bL_X,\bL_Y]-\bL_{[X,Y]}\,\in\End_\RR(\frg/\frh).
\end{equation}

\begin{remark} Equation \eqref{eq:RXhYh} shows that the connection $\nabla$ is flat (i.e., its curvature is trivial) if and only if the map $\bL:\frg\rightarrow \frgl(\frg/\frh)$ is a homomorphism of Lie algebras, that is,  $\bL$ gives a representation of $\frg$ in the quotient space $\frg/\frh$.
\end{remark}

\bigskip

\section{Nomizu's Theorem}\label{se:Nomizu}

\subsection{Reductive homogeneous spaces} 
\begin{definition}
The homogeneous space $M\simeq G/H$ is said to be \emph{reductive} if there is a decomposition as a direct sum of vector spaces $\frg= \frh\oplus\frmm$ such that $\Ad_h(\frmm)\subseteq \frmm$ for any $h\in H$. 
(Such decomposition is called a \emph{reductive decomposition}.)
\end{definition}

\begin{remark}
Because of \eqref{eq:AdtXad}, the condition $\Ad_H(\frmm)\subseteq \frmm$ implies $[\frh,\frmm]\subseteq \frmm$, and the converse holds if $H$ is connected.
\end{remark}

A big deal of information on the reductive homogneous space $M\simeq G/H$ is located in the following two products defined on $\frmm\, (\simeq \frg/\frh\simeq T_pM)$:
\begin{equation}\label{eq:binter}
\begin{split}
\textrm{Binary product:}&\ X\cdot Y\bydef [X,Y]_\frmm\\
  &\qquad\textrm{(the projection of $[X,Y]$ on $\frmm$)}\\[6pt]
\textrm{Ternary product:}&\ [X,Y,Z]\bydef [[X,Y]_\frh,Z]\\
 &\textrm{(where $[X,Y]_\frh$ denotes the projection of $[X,Y]$ on $\frh$).}
\end{split}
\end{equation}

\begin{exercise}\label{ex:LY}
Let $\frg$ be a Lie algebra, $\frh$ a subalgebra and $\frmm$ a subspace such that $\frg=\frh\oplus\frmm$ and $[\frh,\frmm]\subseteq \frmm$.
\begin{enumerate}
\item 
Prove the following properties:
\begin{description}
\item[(LY1)] $X\cdot X=0$,
\item[(LY2)] $[X,X,Y]=0$,
\item[(LY3)] $\sum_{(X,Y,Z)}\Bigl([X,Y,Z]+(X\cdot Y)\cdot Z\Bigr)=0$,
\item[(LY4)] $\sum_{(X,Y,Z)}[X\cdot Y,Z,T]=0$,
\item[(LY5)] $[X,Y,U\cdot V]=[X,Y,U]\cdot V+U\cdot [X,Y,V]$,
\item[(LY6)] $[X,Y,[U,V,W]]=[[X,Y,U],V,W]+[U,[X,Y,V],W]$\newline \null\hfill $+[U,V,[X,Y,W]]$,
\end{description}
for any $X,Y,Z,T,U,V,W\in\frmm$, where $\sum_{(X,Y,Z)}$ stands for the cyclic sum on $X,Y,Z$.

A vector space endowed with a binary and a ternary multilinear products satisfying these properties is called a \emph{Lie-Yamaguti algebra}. These algebras were named `general Lie triple systems' by Yamaguti \cite{Yam58}. The name `Lie-Yamaguti algebras' was given by Kinyon and Weinstein \cite{KW01}. Irreducible Lie-Yamaguti algebras have been studied in \cite{BEM1,BEM2}.

\smallskip

\item Prove that if $(\frmm,\cdot,[...])$ is a Lie-Yamaguti algebra, then there is a Lie algebra $\frg$ containing $\frmm$ and a subalgebra $\frh$ of $\frg$ complementing $\frmm$ with $[\frh,\frmm]\subseteq \frmm$ such that $X\cdot Y=[X,Y]_\frmm$ and $[X,Y,Z]=[[X,Y]_\frh,Z]$ for any $X,Y,Z\in\frmm$.

\noindent{\small [Hint: Let $\frh$ be the subspace of $\End(\frmm)$ spanned by the endomorphismsm $[X,Y,.]$ for $X,Y\in\frmm$, and define $\frg=\frh\oplus\frmm$ with a suitable bracket.]}
\end{enumerate}
\end{exercise}

\smallskip

\subsection{Invariant affine connections on reductive homogeneous spaces}

Let $M\simeq G/H$ be a reductive homogeneous space with reductive decomposition $\frg=\frh\oplus\frmm$. In this case we will identify the quotient space $\frg/\frh$ with $\frmm$. Hence any linear map $\bL:\frg\rightarrow \End_\RR(\frg/\frh)$ satisfying properties (a) and (b) in Proposition \ref{pr:properties_ab} is determined by the linear map
\[
\begin{split}
\tilde \bL: \frmm &\longrightarrow \End_\RR(\frmm)\\
X&\mapsto \begin{aligned}[t]
             \tilde \bL_X:\frmm &\rightarrow \frmm \\
              Y&\mapsto \tilde\phi^{-1}\bigl(\bL_{X^+}(Y^+_p)\bigr),
            \end{aligned}
\end{split}
\]
where $\tilde\phi:\frmm\rightarrow T_pM$ is the composition of the identification $\frmm\simeq\frg/\frh$ and the map $\phi:\frg/\frh\rightarrow T_pM$ in \eqref{eq:ghTpM}.

Note that the restriction of the map $\bL$ in \eqref{eq:L} to $\frh$ is determined by property (b) in Proposition \ref{pr:properties_ab}, while property (a) translates, thanks to \eqref{eq:jh}, into the condition
\begin{equation}\label{eq:tildeLhX}
\tilde \bL_{\Ad_h X}=\left.\Ad_h\right|_\frmm\circ\tilde \bL_X\circ\left.\Ad_{h^{-1}}\right|_\frmm
\end{equation}
for any $h\in H$ and $X\in\frmm$.

But any such map $\tilde \bL$ determines (and is determined by) the bilinear map
\[
\begin{split}
\alpha:\frmm\times\frmm&\longrightarrow \frmm\\
(X,Y)&\mapsto \alpha(X,Y)\bydef \tilde \bL_X(Y),
\end{split}
\]
and Equation \ref{eq:tildeLhX} is equivalent to the condition
\[
\alpha\bigl(\Ad_hX,\Ad_hY)=\Ad_h\bigl(\alpha(X,Y)\bigr)
\]
for any $h\in H$ and $X,Y\in\frmm$.

This proves the following theorem of Nomizu, which is the main result of this course:

\bigskip

\noindent\fbox{\begin{minipage}{0.98\textwidth}
\begin{theorem}[Nomizu, 1954]\label{th:Nomizu}
The invariant affine connections on a reductive homogeneous space $M\simeq G/H$ with reductive decomposition $\frg=\frh\oplus\frmm$ are in bijection with the vector space of nonassociative multiplications $\alpha:\frmm\times\frmm\rightarrow \frmm$ such that $H$ acts by automorphisms, i.e., $\left.\Ad_H\right|_\frmm\subseteq \Aut(\frmm,\alpha)$.
\end{theorem}
\end{minipage}}

\bigskip

\begin{remark}
The vector space $\frmm$ is a module for the group $H$, and the vector space of nonassociative multiplications as in the Theorem above is naturally isomorphic to $\Hom_H(\frmm\otimes_\RR\frmm,\frmm)$. 

Moreover, the condition $\left.\Ad_H\right|_\frmm\subseteq \Aut(\frmm,\alpha)$ implies that $\left.\ad_\frh\right|_\frmm$ is contained in $\Der(\frmm,\alpha)$. The converse is valid if $H$ is connected.
\end{remark}

\begin{remark}
Nomizu \cite{Nom54} proved this result in a very different way. He defined the product $\alpha(X,Y)$ by extending locally $X$ and $Y$ to some invariant vector fields $\tilde X$ and $\tilde Y$ defined on a neighborhood of $p$ and imposing $\alpha(X,Y)\bydef\nabla_{\tilde X_p}\tilde Y\in T_pM\simeq \frmm$.
\end{remark}

The torsion and curvature tensors are determined in equations \eqref{eq:TXhYh} and \eqref{eq:RXhYh}, which now become simpler:
\begin{equation}\label{eq:TRm}
\begin{split}
T(X,Y)&=\alpha(X,Y)-\alpha(Y,X)-X\cdot Y,\\
R(X,Y)Z&=\alpha(X,\alpha(Y,Z))-\alpha(Y,\alpha(X,Z))-\alpha(X\cdot Y,Z)-[X,Y,Z],
\end{split}
\end{equation}
for any $X,Y,Z\in\frmm$, where $X\cdot Y=[X,Y]_\frmm$ and $[X,Y,Z]=[[X,Y]_\frh,Z]$ as in \eqref{eq:binter}.

\medskip

There are always two distinguished invariant affine connections in this case:
\begin{itemize}
\item \textbf{Natural connection} (or canonical connection of the first kind), given by
\[
\alpha (X,Y)=\frac{1}{2}X\cdot Y,
\]
for $X,Y\in\frmm$, which is symmetric ($T(X,Y)=0$ for any $X,Y\in\frmm$).

\smallskip

\item \textbf{Canonical connection} (of the second kind), given by $\alpha(X,Y)=0$ for any $X,Y\in\frmm$.
\end{itemize}

\begin{remark}
These two connections coincide if and only if $X\cdot Y=0$ for any $X,Y\in\frmm$, and this is equivalent to the fact that the reductive decomposition $\frg=\frh\oplus\frmm$ be a $\ZZ/2\ZZ$-grading: $\frg\subo=\frh$ is a subalgebra, $\frg\subuno=\frmm$ is a $\frg\subo$-module, and $[\frg\subuno,\frg\subuno]\subseteq \frg\subo$.

This is the case for the \emph{symmetric spaces}, which are the coset spaces $G/H$ for a Lie group $G$ endowed with an order $2$ automorphism $\sigma$ such that 
\[
G_0^\sigma\subseteq H\subseteq G^\sigma,
\]
where $G^\sigma=\{g\in G: \sigma(g)=g\}$ is the subgroup of fixed elements by $\sigma$ and $G_0^\sigma$ is the connected component of $G^\sigma$ containing the neutral element. In this case $\frh=\{X\in\frg: \sigma_*(X)=X\}$ is the Lie algebra of $H$, and $\frmm\bydef\{X\in\frg: \sigma_*(X)=-X\}$ is a complementary subspace. Since $(\sigma_*)^2=(\sigma^2)_*$ is the identity, and $\left.\sigma_*\right|_\frg$ is an automorphism, this gives a reductive decomposition $\frg=\frh\oplus\frmm$, which is clearly a $\ZZ/2\ZZ$-grading. (Gradings on Lie algebras are the subject of \cite{EKmon}.)

The invariant affine connections on symmetric spaces, and the associated algebras, have been studied in \cite{LaquerSym,BDESym}.
\end{remark}

\medskip

Nomizu's Theorem allows us to transfer geometric conditions to algebra, or algebraic conditions to geometry.

For example, given a homogeneous space $M\simeq G/H$, and a vector field $X\in\frg$, the curve $\gamma: t\mapsto \exp(tX)\cdot p$ is a geodesic with respect to an invariant affine connection $\nabla$ if and only if 
\[
\nabla_{X^+_{\exp(tX)\cdot p}}X^+=0
\]
for any $t\in\RR$. (Recall that $\dot\gamma(t)=X^+_{\exp(tX)\cdot p}$ by the definition of $X^+$.) But
\[
\nabla_{X^+_{g\cdot p}}Y^+=\tau(g)_*\left(\nabla_{(\Ad_{g^{-1}}X)^+_p}(\Ad_{g^{-1}}Y)^+\right)
\]
by Proposition \ref{pr:properties_X+}, since $\nabla$ is invariant. From
$\Ad_{\exp(-tX)}X=X$ (see \eqref{eq:AdtXad}), we check that $\nabla_{X^+_{\exp(tX)\cdot p}}X^+=0$ for any $t\in\RR$ if and only if $\nabla_{X^+_p}X^+=0$, or $\bL_{X^+}(X^+_p)=0$. 

The last condition is equivalent to $\bL_X(X+\frh)=0$, where $\bL:\frg\rightarrow \End_\RR(\frg/\frh)$ is the linear map in \eqref{eq:L}. Hence, if $M\simeq G/H$ is reductive and $\alpha$ is the multiplication on $\frmm$ determined by $\nabla$ we obtain the following result:

\begin{proposition}
Under the hypotheses above, the multiplication $\alpha$ attached to $\nabla$ is anticommutative (i.e., $\alpha(X,X)=0$ for any $X\in\frmm$) if and only if the geodesics through $p$ are exactly the curves $t\mapsto \exp(tX)\cdot p$ for $X\in\frmm$.
\end{proposition}

Note that given a tangent vector $v\in T_pM$ there is a unique geodesic $\mu(t)$ through $p$: $\mu(0)=p$, such that $\dot\mu(0)=v$ (in general, $\mu$ is defined only locally). The geodesics through other points are given by `translation'. Actually, if $q=g.p$, then $\tau(g^{-1})_*\nabla_{X_q^+}X^+=\nabla_{\tau(g^{-1})_*(X^+_q)}\tau(g^{-1})_*(X^+)=\nabla_{(\Ad_{g^{-1}}X)^+_p}(\Ad_{g^{-1}}X)^+_p$, so if $\nabla_{X^+_p}X^+=0$ for any $X\in \frg$, also $\nabla_{X^+_q}X^+=0$ for any  $X\in\frg$ and $q\in M$.

\begin{corollary} 
The natural connection is the only symmetric invariant affine connection such that the curves $t\mapsto \exp(tX)\cdot q$ are geodesic for any $X\in\frmm$ and $q\in M$.
\end{corollary}
\begin{proof} 
If $\alpha$ is anticommutative, then we have $T(X,Y)=2\alpha(X,Y)-X\cdot Y$. So the connection is symmetric (zero torsion) if and only if $\alpha(X,Y)=\frac{1}{2}X\cdot Y$.
\end{proof}

\begin{exercise} 
Let $M\simeq G/H$ be a homogeneous space endowed with an invariant pseudo-metric $g$. This means that $M$ is pseudo-Riemannian with the pseudo-metric $g$ satisfying
\[
g(\tau(x)_*(u),\tau(x)_*(v))_{x\cdot q}=g(u,v)_q
\]
for any $x\in G$ and $u,v\in T_qM$. Hence, for any $q\in M$, $X\in\frg$, and $U,V\in\cX(M)$ we have:
\[
\begin{split}
X^+_q\bigl(g(U,V)\bigr)
 &=\left.\frac{d\ }{dt}\right|_{t=0}g(U,V)_{\exp(tX)\cdot q}\\
 &=\left.\frac{d\ }{dt}\right|_{t=0}g\bigl(\tau(\exp(-tX))_*(U),\tau(\exp(-tX))_*(V)\bigr)_q\\
 &=g\bigl([X^+,U],V\bigr)_q+g\bigl(U,[X^+,V]\bigr)_q\quad\textrm{(see \eqref{eq:XYPhi}).}
\end{split}
\]
In particular, for $X\in\frh$, $X^+_p=0$, so $g\bigl([X^+,U],V\bigr)_p + g\bigl(U,[X^+,V]\bigr)_p=0$.
\begin{enumerate}
\item Prove that the Levi-Civita connection of $(M,g)$ is invariant.
\item Check that the Nomizu operator $\bL_{X^+}$ for $X\in\frg$ satisfies
\[
g\bigl(\bL_{X^+}(U),V)\bigr)+g\bigl(U,\bL_{X^+}(V)\bigr)=0
\]
for any $U,V\in\cX(M)$.
\item Conclude that for $X,Y,Z\in\frg$ we have
\[
2g\bigl(\bL_{X^+}(Y^+),Z^+\bigr)=
g\bigl([X,Y]^+,Z^+\bigr)-g\bigl([X,Z]^+,Y^+\bigr)-g\bigl(X^+,[Y,Z]^+\bigr).
\]
\item Prove that if $\bL:\frg\rightarrow\End_\RR(\frg/\frh)$ is the associated linear map in Theorem \ref{th:nablaL} and we identify $g_p:T_pM\times T_pM\rightarrow\RR$ with a bilinear map
\[
\bar g:\frg/\frh\times\frg/\frh\rightarrow\RR
\]
by means of \eqref{eq:ghTpM}, then for any $X,Y,Z\in\frg$ we get
\begin{multline*}
2\bar g\bigl(\bL_X(Y+\frh),Z+\frh\bigr)\\
=\bar g\bigl([X,Y]+\frh,Z+\frh\bigr)-\bar g\bigl(Y+\frh,[X,Z]+\frh\bigr)
 -\bar g\bigl(X+\frh,[Y,Z]+\frh\bigr).
\end{multline*}
\item Deduce that if $M$ is reductive with reductive decomposition $\frg=\frh\oplus\frmm$, and we identify $\bar g$ with a bilinear map (also denoted by $\bar g$) $\frmm\times\frmm\rightarrow\RR$, the multiplication $\alpha$ in $\frmm$ attached by Nomizu's Theorem \ref{th:Nomizu} with the Levi-Civita connection on $M$ is determined by the equation:
\[
2\bar g\bigl(\alpha(X,Y),Z\bigr)=\bar g(X\cdot Y,Z)-\bar g(X\cdot Z,Y)-\bar g(X,Y\cdot Z)
\]
for any $X,Y,Z\in\frmm$. Conclude from (2) that
\[
\bar g(\alpha(X,Y),Z)+\bar g(Y,\alpha(X,Z))=0
\]
for any $X,Y,Z\in\frmm$.

\item Since the Levi-Civita connection is symmetric, $\alpha(X,Y)=\frac{1}{2}X\cdot Y+ \mu(X,Y)$, where $\mu(X,Y)$ is a commutative multiplication in $\frmm$. Prove that this commutative multiplication is determined by the equation:
\[
2\bar g\bigl(\mu(X,Y),Z\bigr)=\bar g(Z\cdot X,Y)+\bar g(X,Z\cdot Y)
\]
for any $X,Y,Z\in\frmm$.
\item Prove that the Levi-Civita connection coincides with the natural connection if and only if
\[
\bar g(Y\cdot X,Z)=\bar g(Y,X\cdot Z)
\]
for any $X,Y,Z\in\frmm$ (i.e., $\bar g$ is associative relative to the product $X\cdot Y$). In this case $M$ is said to be a \emph{naturally reductive homogeneous manifold} (see \cite{Nagai97}).
\end{enumerate}
\end{exercise}

\bigskip

\section{Examples}\label{se:examples}

\subsection{Left invariant affine connections on Lie groups}
Any Lie group $G$ is a reductive homogeneous space under the action of $G$ into itself by left multiplication. The isotropy subgroup at any point is trivial. Hence $G$ is a reductive homogeneous space with $\frh=0$ and $\frmm=\frg$. Also, for any $X\in\frg$, the vector field $X^+$ is nothing else but the right invariant vector field $\hat X\in\frg_\textrm{right}$ with $\hat X_e=X_e$ (see Proposition \ref{pr:XXhat}).

According to Nomizu's Theorem (Theorem \ref{th:Nomizu}) we have a bijection
\begin{align*}
\left\{\begin{matrix} \textrm{Left invariant affine}\\
\textrm{connections on $G$}\end{matrix}\right\}
&\longleftrightarrow
\left\{\begin{matrix}\textrm{bilinear maps}\\
\alpha:\frg\times\frg\rightarrow\frg\end{matrix}\right\}\\
\nabla\qquad &\mapsto\qquad \alpha(X,Y)\bydef \phi^{-1}\Bigl((\bL_{X^+}(Y^+)_e\Bigr)
\end{align*}
where $\phi:\frg\rightarrow T_eM$, $X\mapsto X_e$, as in \eqref{eq:ghTpM}.

For $X,Y\in\frg$ we get
\[
\bigl(\bL_{X^+}(Y^+)\bigr)_e=\bL_{\hat X}(\hat Y_e)=\bL_{\hat X}(Y_e)=\bigl(\bL_{\hat X}(Y)\bigr)_e=\Bigl( \nabla_{\hat X}Y-[\hat X,Y]\Bigr)_e,
\]
and
\[
\bigl(\nabla_{\hat X}Y\bigr)_e=\nabla_{\hat X_e}Y=\nabla_{X_e}Y=\bigl(\nabla_XY\bigr)_e.
\]
Besides, $\nabla_XY\in \frg$ by left invariance, and $[\hat X,Y]=0$ since the corresponding flows commute: the flow of $\hat X$ is $L_{\exp(tX)}$ and the flow of $Y$ is $R_{\exp(tY)}$ and left and right multiplications commute by associativity of the multiplication in the group $G$. 

Here we use the following result:

\begin{lemma}
Let $M$ be a manifold, $X,Y\in\cX(M)$ with flows $\Phi_t$ and $\Psi_t$ respectively. 
If $\Phi_t\circ\Psi_s=\Psi_s\circ\Phi_t$  for some 
$\epsilon>0$ and $-\epsilon<s,t<\epsilon$, then $[X,Y]=0$.
\end{lemma}
\begin{proof}
Just notice that for any $f\in\cC^\infty(M)$ and any point $q\in M$
\[
\begin{split}
[X,Y]_q(f)&=\left.\frac{d\ }{dt}\right|_{t=0}\Bigl((\Phi_{-t})_*(Y)\Bigr)_q(f)
 =\left.\frac{d\ }{dt}\right|_{t=0}Y_{\Phi_t(q)}(f\circ\Phi_{-t})\\
 &=\left.\frac{d\ }{dt}\right|_{t=0}\left(\left.\frac{d\ }{ds}\right|_{s=0}
 \bigl(f\circ\Phi_{-t}\circ\Psi_s(\Phi_t(q))\bigr)\right)\\
 &=\left.\frac{d\ }{dt}\right|_{t=0}\left(\left.\frac{d\ }{dt}\right|_{t=0}
 \bigl(f\circ\Phi_s(q)\bigr)\right)
 \quad\textrm{(because $\Phi_{-t}$ and 
 $\Psi_s$ commute)}\\
 &=\left.\frac{d\ }{dt}\right|_{t=0}\bigl(Y_q(f)\bigr)=0
 \quad\textrm{(as $Y_q(f)$ does not depend on $t$).\qedhere}
\end{split}
\]
\end{proof}

\medskip

Therefore $\bigl(\bL_{X^+}(Y^+)\bigr)_e=\bigl(\nabla_XY\bigr)_e$ and $\phi^{-1}\bigl((\bL_{X^+}(Y^+)_e\bigr)=\nabla_XY$. Hence the bijection in Nomizu's Theorem becomes a very natural one:
\begin{equation}\label{eq:Nomizu_left}
\begin{aligned}
\left\{\begin{matrix} \textrm{Left invariant affine}\\
\textrm{connections on $G$}\end{matrix}\right\}
&\longleftrightarrow
\left\{\begin{matrix}\textrm{bilinear maps}\\
\alpha:\frg\times\frg\rightarrow\frg\end{matrix}\right\}\\[4pt]
\nabla\qquad &\mapsto\qquad \alpha(X,Y)=\nabla_XY
\end{aligned}
\end{equation}

\begin{remark}\label{re:Nomizu_left}
\begin{itemize} 
\item The canonical connection is then the connection that satisfies $\nabla_XY=0$ for any $X,Y\in \frg$. That is, it is the affine connection in which the left invariant vector fields are parallel.

\item A left invariant affine connection $\nabla$ on a Lie group $G$ is symmetric if and only if $\nabla_XY-\nabla_YX=[X,Y]$ for any $X,Y\in\frg$. This means that the algebra $(\frg,\alpha)$ (with multiplication $\alpha(X,Y)=\nabla_XY$) is \emph{Lie-admissible} with associated Lie algebra $\frg$: 

A nonassociative algebra $(A,\alpha)$ is said to be \emph{Lie-admissible} if the new algebra defined on $A$ but with multiplication $\alpha^-(x,y)\bydef \alpha(x,y)-\alpha(y,x)$ is a Lie algebra (the associated Lie algebra). For instance,  associative algebras and  Lie algebras are Lie-admissible.

\item A left invariant affine connection $\nabla$ on a Lie group $G$ is flat if and only if $[\nabla_X,\nabla_Y]=\nabla_{[X,Y]}$ for any $X,Y\in\frg$. 
\end{itemize}
\end{remark}

\begin{exercise}\label{exe:leftsymmetric}
A nonassociative algebra $(A,\alpha)$ is \emph{left-symmetric} if the associator, defined as the trilinear map $(x,y,z)\bydef \alpha(\alpha(x,y),z)-\alpha(x,\alpha(y,z))$ (the associator measures the lack of associativity of $\alpha$!), is symmetric in the first two components. 
\begin{enumerate}
\item
Prove that any left-symmetric algebra is Lie-admissible.

\item Deduce that a left invariant affine connection on a Lie group $G$ is symmetric and flat if and only  if the corresponding nonassociative algebra $(\frg,\alpha)$ (where $\alpha(X,Y)=\nabla_XY$) is left-symmetric and the algebra $(\frg,\alpha^-)$ is just the Lie algebra $\frg$. (Therefore, the left invariant symmetric (i.e., torsion-free) and flat affine connections on a Lie group are in a natural bijection to the set of left-symmetric algebra structures on the corresponding Lie algebra.)
\end{enumerate}
\end{exercise}

The reader may consult the survey paper \cite{Burde} for more results on (and applications of) left-symmetric algebras.

\bigskip

\subsection{Bi-invariant affine connections on Lie groups}

Let $\nabla$ be a left invariant affine connection on a Lie group $G$. Then $\nabla$ is also right invariant if for any $g\in G$ and $X,Y\in\frg$, 
\[
(R_{g^{-1}})_*(\nabla_XY)=\nabla_{(R_{g^{-1}})_*(X)}(R_{g^{-1}})_*(Y).
\]
But by left invariance $(R_{g^{-1}})_*(X)=(R_{g^{-1}})_*\circ (L_g)_*(X)=\Ad_gX$, and similarly for $Y$ and $\nabla_XY$. Hence $\nabla$ is bi-invariant if and only if
\[
\Ad_g(\nabla_XY)=\nabla_{\Ad_gX}(\Ad_gY)
\]
for any $g\in G$ and $X,Y\in\frg$. Therefore the bijection in \eqref{eq:Nomizu_left} restricts to a bijection
\begin{equation}\label{eq:Nomizu_bi}
\begin{aligned}
\left\{\begin{matrix} \textrm{Bi-invariant affine}\\
\textrm{connections on $G$}\end{matrix}\right\}
&\longleftrightarrow
\left\{\begin{matrix}\textrm{bilinear maps}\\
\alpha:\frg\times\frg\rightarrow\frg\\
\textrm{with $\Ad_G\subseteq \Aut(\frg,\alpha)$}\end{matrix}\right\}\\[4pt]
\nabla\qquad &\mapsto\qquad \alpha(X,Y)=\nabla_XY.
\end{aligned}
\end{equation}
Note that the space on the right can be canonically identified with $\Hom_G(\frg\otimes_\RR\frg,\frg)$, which is contained in $\Hom_\frg(\frg\otimes_\RR\frg,\frg)$ (and they coincide if $G$ is connected).

\begin{remark}
\begin{itemize} 
\item If $\nabla$ is a symmetric bi-invariant affine connection on a Lie group $G$, and $\alpha:\frg\times\frg\rightarrow\frg$ is the associated multiplication ($\alpha(X,Y)=\nabla_XY$), then we know from Remark \ref{re:Nomizu_left} that $(\frg,\alpha)$ is Lie-admissible with associated Lie algebra $\frg$. But here $\alpha$ lies in $\Hom_\frg(\frg\otimes_\RR\frg,\frg)$. This means that 
\[
[X,\alpha(Y,Z)]=\alpha([X,Y],Z)+\alpha(Y,[X,Z])
\] 
for any $X,Y,Z\in\frg$. That is, $\ad_X$ is a derivation of $(\frg,\alpha)$. Then the algebra $(\frg,\alpha)$ is flexible and Lie-admissible.
(See Exercise \ref{exe:flexibleLA}). Flexible Lie-admissible algebras are studied in \cite{Myung}.

\item If $\nabla$ is a symmetric and flat bi-invariant affine connection on a Lie group $G$, and $\alpha:\frg\times\frg\rightarrow\frg$ is the associated multiplication: $\alpha(X,Y)=\nabla_XY$, then we know from Remark \ref{re:Nomizu_left} that $(\frg,\alpha)$ is left-symmetric and flexible, but then the algebra $(\frg,\alpha)$ is associative \cite{Medina81}. (See Exercise \ref{exe:flexibleLA}.)
\end{itemize}
\end{remark}

\begin{exercise}\label{exe:flexibleLA}
Let $(A,\alpha)$ be a nonassociative algebra. By simplicity, write $\alpha(x,y)=xy$ for any $x,y\in A$. Recall from Exercise \ref{exe:leftsymmetric} that the associator of the elements $x,y,z\in A$ is the element $(x,y,z)=(xy)z-x(yz)$. The algebra $(A,\alpha)$ is said to be \emph{flexible} if $(x,y,x)=0$ for any $x,y\in A$. That is, the associator is alternating in the first and third arguments. 
\begin{enumerate}
\item Assume that this algebra is Lie-admissible (i.e.; the bracket $[x,y]\bydef xy-yx$ satisfies the Jacobi identity). Prove that it is flexible if and only if the map $\ad_x:y\mapsto [x,y]$ is a derivation of $(A,\alpha)$ for any $x\in A$: $[x,yz]=[x,y]z+y[x,z]$ for any $x,y,z\in A$.

\item Check that the Lie algebra $\frsu(n)$ of the special unitary group $SU(n)$, endowed with the multiplication given by:
\[
\alpha(X,Y)=\mu XY-\bar \mu YX-\frac{\mu-\bar\mu}{n}\trace(XY),
\]
where $\mu\in\CC$ satisfies $\mu+\bar \mu=1$, is a flexible and Lie-admissible algebra.\\
This example plays a key role in the classification of the invariant connections on compact simple Lie groups by Laquer \cite{Laquer}.

\item Prove that an algebra $(A,\alpha)$ is flexible and left-symmetric if and only if it is associative.
\end{enumerate}
\end{exercise}

\medskip

Simple non-Lie Malcev algebras \cite{EM_S7}, and other interesting nonassociative algebras \cite{BDE_G2,DGP,EM_S15,EMColor,Sagle_Jordan}, appear too related to some specific reductive homogeneous spaces. 

\medskip

Nomizu's Theorem has been extended in order to study invariant affine connections on principal fibre bundles which admits a fibre-transitive Lie group of automorphisms \cite{Wang}, but this is out of the scope of this introductory survey.


\end{document}